\documentclass[]{article}
\usepackage{graphicx}
\usepackage{amsfonts}
\usepackage{amsmath}
\usepackage{amssymb}
\usepackage{fancyhdr}
\usepackage{titlesec}
\usepackage{indentfirst}
\usepackage{booktabs}
\usepackage{verbatim}
\usepackage{color}
\usepackage{amsthm}
\usepackage{subfigure}
\usepackage{caption}

\usepackage{hyperref}
\usepackage[page,header]{appendix}
\usepackage{titletoc}

\numberwithin{equation}{section}
\newtheorem{theorem}{Theorem}[section]

\newtheorem{proposition}[theorem]{Proposition}

\newtheorem{remark}[theorem]{Remark}

\newcommand{\rd}[1]{\,\mathrm{d}#1}
\newcommand{\RN}[1]{
  \textup{\uppercase\expandafter{\romannumeral#1}}
}

\begin{document}

\title{Explicit charge-conserving and modified energy-conserving PIC methods using the SAV approach}
\author{Chengyu Fan\footnote{Department of Applied Mathematics, University of Washington, Seattle, WA 98195 (whitefcy@uw.edu).}, \ Nicolas Crouseilles\footnote{Inria Rennes (Mingus team) \& IRMAR, UMR CNRS 6625, Univ Rennes, Rennes F-35042, France ENS Rennes (nicolas.crouseilles@inria.fr).}, \  Jingwei Hu\footnote{Department of Applied Mathematics, University of Washington, Seattle, WA 98195 (hujw@uw.edu).}}     
\date{}

\maketitle

\begin{abstract}
We develop explicit second-order particle-in-cell (PIC) methods for the Vlasov--Amp\`ere and Vlasov--Maxwell systems using the scalar auxiliary variable (SAV) approach. Two schemes are proposed for each system: a predictor-based scheme and a symmetric splitting scheme using a scaled electric field. The methods exactly preserve the discrete Gauss law and a modified total energy. Numerical experiments for electrostatic and electromagnetic problems demonstrate the expected dynamics as well as the conservation properties of both invariants.
\end{abstract}

{\small 

{\bf AMS subject classifications.} 
} 65M75, 35Q83, 82D10

\section{Introduction}
Particle-in-cell (PIC) methods are among the most widely used numerical tools for Vlasov equations and kinetic plasma simulations \cite{birdsall-2018, verboncoeur-2005, hockney-1988, grigoryev2012numerical}. By representing the distribution function with particles in phase space and coupling them to electromagnetic fields on a mesh, PIC methods avoid a direct grid discretization of the full phase space and remain relatively simple, robust, and efficient. 

Exact energy conservation and charge conservation, or equivalently the preservation of Gauss's law, are critical for the reliability of numerical results. The discretization inherent in PIC methods introduces numerical errors. In long-time simulations, if energy conservation is violated, the accumulation of these numerical errors and statistical particle noise may lead to unphysical energy growth known as grid instability \cite{birdsall-2018, ricketson2025}. If charge conservation is violated, there may be unphysical field oscillations and filamentary structures \cite{pinto-2014, pinto-2022, stock-2011}. In the case of electromagnetic particle solvers, the preservation of Gauss's law has also been recognized as a key requirement  (see \cite{boris1971relativistic}, \cite{marder-1987}, \cite{villasenor1992}), since spurious virtual charges can arise when the numerical method fails to preserve it; see, e.g., \cite{pinto-2014}, \cite{stock-2011}.

Historically, explicit PIC methods are often momentum-conserving but not energy-conserving, leading to numerical heating. Lewis and Langdon's early energy-conserving PIC methods used the derivative of the discrete electrostatic energy as the electric force to achieve energy conservation, but only in the small timestep limit \cite{lewis-1970, langdon-1973}. Later, as modern linear and nonlinear solvers become efficient, fully implicit or semi-implicit PIC methods are designed to conserve energy exactly for arbitrary time steps \cite{lapenta-2017, chen-2011, markidis-2011}. However, explicit energy-conserving PIC methods are still of great interest as they are much simpler and more efficient. Ricketson and Hu's recent work \cite{ricketson2025} used a velocity correction method to achieve explicitness and energy conservation, but it may suffer from problematic particles caused by imaginary correction factors. 

Charge conservation adds another constraint to the scheme. Early approaches often perform a correction on the electric field to satisfy the discrete Gauss's law \cite{birdsall-2018, marder-1987, langdon-1992}. Later, the Villasenor-Buneman method \cite{villasenor1992} became popular for charge conservation, which modifies the shape function in the Amp\`ere source term to make it compatible with the charge update. Modern extensions of this method accommodate a wider range of shape functions \cite{esirkepov-2001} and achieve higher efficiency \cite{umeda-2003}.

The scalar auxiliary variable (SAV) method was introduced by Shen et al. for dissipative nonlinear systems \cite{shen2018sav}, which reformulates  the free energy into a quadratic form with a new scalar variable to achieve unconditional stability of a modified energy. Recently, Crouseilles et al. adapted this idea to semi-Lagrangian methods of the Vlasov--Amp\`ere (VA) and Vlasov--Maxwell (VM) systems \cite{crouseilles2026}. In their work, the SAV brings an extra degree of freedom into the system, making it possible to conserve the charge and total modified energy while keeping the scheme explicit. 

This paper adopts the SAV approach in the context of particle methods and develops explicit second-order PIC schemes that conserve both charge and a total modified energy for the VA and VM systems. Compared with the local velocity-correction strategy in \cite{ricketson2025}, the modified-energy balance is enforced through an auxiliary scalar variable rather than by correcting each particle velocity. As a result, the method does not encounter imaginary correction factors or  problematic particles. For each system, we introduce two schemes: a predictor-based scheme and a symmetric splitting scheme using a scaled
electric field. 

The rest of the paper is organized as follows. Section 2 presents the electrostatic Vlasov--Amp\`ere system, its SAV reformulation, the PIC semi-discretization, and the fully discrete schemes. Section 3 develops the corresponding  schemes for the electromagnetic Vlasov--Maxwell system in 1D2V setting. Their extension to the 3D3V case is also discussed. Numerical results are reported in Section 4. Finally, Section 5 concludes the paper and discusses future work.

\section{Electrostatic Vlasov-Amp\`ere system}

In this section, we present the VA system, its SAV reformulation, and the proposed numerical schemes. 

The VA system in 1D1V (one dimension in physical space and one dimension in velocity space) is given by
\begin{subequations}
\label{VA}
\begin{align}
\label{VA:Vlasov}
&\partial_t f + v\partial_x f +E\partial_v f = 0, \\
\label{VA:Ampere}
&\partial_t E = -J, \quad J=\int_{\mathbb{R}} v f \rd{v},
\end{align}
\end{subequations}
where $f=f(t,x,v)$ is the distribution function depending on time $t \ge 0$, position $x\in [0, L]$ ($L$ being the domain length), and velocity $v\in \mathbb{R}$. $E(t,x)$ is the electric field and $J(t,x)$ is the current density. We consider periodic boundary conditions in physical space and vanishing boundary conditions in the velocity direction. The initial condition is given by $f(0,x,v)=f_0(x,v)$ and normalized such that $\int_0^L\int_{\mathbb{R}}f_0(x,v)\rd{v}\rd{x}=L$.

The Amp\`ere equation \eqref{VA:Ampere} advances the electric field $E$, whereas Gauss's law
\begin{equation}
\label{VA:Gauss}
\partial_x E=\rho-1,\quad \rho=\int_{\mathbb{R}} f\rd{v} ,
\end{equation}
is a constraint on the initial data, where $\rho$ is the charge density. Indeed, integrating \eqref{VA:Vlasov} with respect to $v$, together with \eqref{VA:Ampere}, gives $\partial_t(\partial_xE-(\rho-1))=0$. Thus, an initially satisfied Gauss's law remains satisfied throughout the VA evolution. In this work, equation \eqref{VA:Gauss} is solved only at $t=0$ to construct a compatible electric field $E(0,x)$, after which $E$ is advanced by \eqref{VA:Ampere}. 

The VA system possesses infinitely many Casimir invariants. In particular, the total energy $\mathcal{E}_T$, defined as the sum of the kinetic and electric energies, 
\begin{equation}
    \mathcal{E}_{T}(t) = \mathcal{E}_K(t) + \mathcal{E}_E(t) = \frac{1}{2}\int_0^L\int_{\mathbb{R}}v^2 f \rd{v}\rd{x} + \frac{1}{2}\int_0^L E^2 \rd{x},
\end{equation}
remains constant over time.

\subsection{SAV formulation and the particle method}
\label{sec:SAV_VA}
To achieve both energy and charge conservation, we use the SAV method \cite{shen2018sav}. We introduce an auxiliary variable $r(t)$ such that $r^2(t)=\frac{1}{2}\int_0^L E^2(t, x) \rd{x}+C_0$, where $C_0$ is a constant chosen to ensure that the right-hand side is always positive. We then look for an equation satisfied by $r$.  
Taking the time derivative gives $\frac{\rd{r}}{\rd{t}} =\frac{1}{2r}\int_0^L E \partial_t E\rd{x} =-\frac{1}{2r}\int_0^L E J\rd{x}$, where we used \eqref{VA:Ampere} in the last equality. The VA system can be written equivalently as
\begin{subequations}
\label{VA_SAV}
\begin{align}
\label{VA_SAV1}
&\partial_t f + v\partial_x f + E \partial_v f = 0, \\
\label{VA_SAV2}
&\partial_t E = -J, \\
\label{VA_SAV3}
&\frac{\rd{r}}{\rd{t}} = - \frac{1}{2r}\int_0^L E J \rd{x}.
\end{align}
\end{subequations}

In the particle method, we approximate the distribution function with a sum of Dirac delta functions
\begin{equation}
    f(t,x,v) \approx\sum_{p=1}^{N_p} \omega_p \delta(x-x_p(t))\delta(v-v_p(t)),
\end{equation}
where $N_p$ is the number of particles, $\omega_p$ is the weight of the $p$-th particle, and $(x_p(t),v_p(t))$ is the position and velocity of the $p$-th particle. For the spatial domain $[0,L]$, we divide it into $N_x$ uniform cells of size $\Delta x=L/N_x$. The electric field $E$ and current density $J$ are defined at cell centers $x_{i+1/2}=(i+1/2)\Delta x$, while the charge density $\rho$ is  defined at cell interfaces $x_i=i\Delta x$, $i=0,\dots,N_x-1$. The grid indices are periodic modulo $N_x$, and the particle positions are periodic modulo $L$. 

The system \eqref{VA_SAV1}-\eqref{VA_SAV3} can be approximated using particles as
\begin{subequations}
\label{VA_SAV_semi-discrete}
\begin{align}
\label{VA_dx}
    \dot{x}_p(t) &= v_p(t), \\
\label{VA_dv}
    \dot{v}_p(t) &= E(x_p(t)), \\
\label{VA_dE}
\dot{E}_{i+1/2}(t) &= -J_{i+1/2}(t), \\
\label{VA_dr}
    \dot{r}(t) &= - \frac{1}{2r(t)}\sum_{i=0}^{N_x-1} E_{i+1/2}(t)J_{i+1/2}(t) \Delta x. 
\end{align}
\end{subequations} 
The initial positions and velocities of the particles, $\{x_p(0),v_p(0)\}_{p=1}^{N_p}$, are sampled randomly from the initial condition $f_0(x,v)$, in which case the particle weight $\omega_p=L/N_p$ is the same for all particles. $r$ is initialized as $r^2(0)=\frac{\Delta x}{2}\sum_{i=0}^{N_x-1} E_{i+1/2}^2(0) + C_0$.  A Gauss-law-compatible initial electric field is constructed using the following discrete Poisson solve:
\begin{equation}
\label{VA_Poisson}
\frac{1}{\Delta x}(E_{i+1/2}-E_{i-1/2})=\rho_i-1, \quad E_{i+1/2}=-\frac{1}{\Delta x}(\phi_{i+1}-\phi_i),
\end{equation}
where for the potential $\phi$, we impose $\phi_{N_x}=\phi_0$ and set $\phi_0=0$ to fix the additive constant. 

We need to define the particle-mesh coupling to transfer quantities between the spatial grid and the particle positions. The charge density $\rho_i$ and current density $J_i$ on the grid are defined as
\begin{equation}
\label{def:rho_J}
    \rho_i = \frac{1}{\Delta x}\sum_{p=1}^{N_p} \omega_p S^m(x_i-x_p), \quad J_{i+1/2} = \frac{1}{\Delta x}\sum_{p=1}^{N_p} \omega_p v_p S^{m}(x_{i+1/2}-x_p).
\end{equation}
The electric field at the particle position $x_p$ is defined as
\begin{equation}
\label{E_part}
	E(x_p) = \sum_{i=0}^{N_x-1} E_{i+1/2} S^m(x_p - x_{i+1/2}).
\end{equation}
Here, $S^m(x)$ is the $B$-spline function defined recursively by
\begin{align}
S^{0}(x) = \begin{cases} 
1 & |x| < \Delta x/2, \\
0 & \text{otherwise},
\end{cases}, 
\;\;\; 
S^{m}(x) = \frac{1}{\Delta x} \int_{x-\Delta x/2}^{x+\Delta x/2} S^{m-1}(y)\rd{y} \;\; \text{for} \ m\geq 1. 
\end{align}

We next recall some properties of the B-splines that will be used in the sequel, in particular for the proof of the charge conservation in the fully discrete schemes.

First, let us note that $S^m(x)$ and $S^{m-1}(x)$ satisfy
\begin{equation}
\label{SmSm-1}
\frac{\rd}{\rd{x}} S^m(x_i - x) =-(S^{m})'(x_i-x) =\frac{1}{\Delta x} \left( S^{m-1}\left(x_{i-1/2} - x\right)- S^{m-1}\left(x_{i+1/2} - x\right)\right).
\end{equation}
Moreover, we define the normalized primitive of $S^{m-1}$: 
\begin{equation}
    \label{Sm-1_prim}
\Phi^{m-1}(x) = \frac{1}{\Delta x}\int_{-\infty}^{x} S^{m-1}(y) \rd{y}, \;\; \mbox{ then } \;\;  \frac{1}{\Delta x}\int_{x_1}^{x_2} S^{m-1}(x)\rd{x}=\Phi^{m-1}(x_2)-\Phi^{m-1}(x_1).
\end{equation}

In this paper we use $m=2$, then the previous general formulas become
\begin{align}
S^{1}(x) = \begin{cases} 
1 - \left|\frac{x}{\Delta x}\right| & |x| < \Delta x, \\
0 & \text{otherwise},
\end{cases}, \quad
S^2(x) = \begin{cases} 
\frac{3}{4} - \left(\frac{x}{\Delta x}\right)^2 & |x| \leq \frac{\Delta x}{2}, \\
\frac{1}{2} \left( \frac{3}{2} - \left|\frac{x}{\Delta x}\right| \right)^2 & \text{for } \frac{\Delta x}{2} < |x| \leq \frac{3\Delta x}{2}, \\
0 & \text{otherwise},
\end{cases}
\end{align}
and
\begin{equation}
	\Phi^1(x) = \frac{1}{\Delta x}\int_{-\infty}^{x} S^1(y) \rd{y} = 
\begin{cases} 
0 & x < -\Delta x, \\
\frac{1}{2}(\frac{x}{\Delta x} + 1)^2 & -\Delta x \le x \le 0, \\
1 - \frac{1}{2}(1 - \frac{x}{\Delta x})^2 & 0 < x \le \Delta x, \\
1 & x > \Delta x.
\end{cases}
\end{equation}

\subsection{Scheme VA-I}
\label{subsec: scheme VA-I}

Having defined the spatial discretization, we now introduce the first fully discrete SAV scheme for \eqref{VA_SAV_semi-discrete}, VA-I, which uses a predictor and is second order accurate in time: 
\begin{subequations}
\begin{align}
    v_p^*-v_p^n&=\frac{\Delta t}{2} E^n(x_p^n),\\
    x_p^{n+1}- x_p^n&= \Delta t v_p^{*},\\
    \label{VA1_E}
    E_{i+1/2}^{n+1} -E_{i+1/2}^n&= -\Delta t \bar{J}_{i+1/2},\\
    \label{VA1_v}
  v_p^{n+1} -  v_p^n&= \Delta t E^{n+1/2}(x_p^{n+1/2}),\\
  \label{VA1_r}
  r^{n+1} - r^n &= -\frac{\Delta t \Delta x}{2 r^{n+1/2}}\sum_{i=0}^{N_x-1} E_{i+1/2}^{n+1/2} J_{i+1/2}^{n+1/2},
 \end{align} 
\end{subequations}
 
where
\begin{align}
E^{n}(x_p^{n})&=\sum_{i=0}^{N_x-1} E_{i+1/2}^{n}S^m(x_p^{n}-x_{i+1/2}),\\
\label{cur_VA-I}
    \bar{J}_{i+1/2}&=\frac{1}{\Delta t}\sum_{p=1}^{N_p} \omega_p [\Phi^{m-1}(x_{i+1/2}-x_p^n)-\Phi^{m-1}(x_{i+1/2}-x_p^{n+1})],\\
    E^{n+1/2}(x_p^{n+1/2})&=\sum_{i=0}^{N_x-1} E_{i+1/2}^{n+1/2}S^m(x_p^{n+1/2}-x_{i+1/2}),\\
    J_{i+1/2}^{n+1/2}&=\frac{1}{\Delta x} \sum_{p=1}^{N_p} \omega_p v_p^{n+1/2} S^m(x_{i+1/2}-x_p^{n+1/2}),
\end{align}
and
\begin{align}
x_p^{n+1/2}=\frac{x_p^n+x_p^{n+1}}{2}, \quad v_p^{n+1/2}=\frac{v_p^n+v_p^{n+1}}{2},\quad r^{n+1/2}=\frac{r^n+r^{n+1}}{2}, \quad E_{i+1/2}^{n+1/2}=\frac{E_{i+1/2}^n+E_{i+1/2}^{n+1}}{2}.
\end{align}

Note that the semi-discrete system \eqref{VA_SAV_semi-discrete} uses $J_{i+1/2}$ in both the
electric-field and SAV updates. 
After time discretization, however, two different approximations of the current are used: a Villasenor-Buneman style current $\bar{J}_{i+1/2}$ \cite{villasenor1992} is used in the electric-field update, which ensures charge conservation (as shown below), and a standard approximation of the current from the particles is used for the $r$ update.

We first justify that $\bar{J}_{i+1/2}$ is still a second-order in time approximation to the exact current.
Consider the time interval $[t^n,t^{n+1}]$ and let $(x_p(t),v_p(t))$ denote the exact quantities with the initial data at $t^n$:
\begin{equation}
\dot{x}_p(t)=v_p(t),\quad
\dot{v}_p(t)=E(x_p(t)),\quad
x_p(t^n)=x_p^n,\quad v_p(t^n)=v_p^n.
\end{equation} 
Define the exact current $\tilde{J}_{i+1/2}(t) = \frac{1}{\Delta x}\sum_{p=1}^{N_p} \omega_p v_p(t) S^{m-1}(x_{i+1/2}-x_p(t))$. 
Since $(\Phi^{m-1})'(x)=\frac{1}{\Delta x}S^{m-1}(x)$, we have 
\begin{align}
\int_{t^n}^{t^{n+1}} \tilde J_{i+1/2}(t)\rd t &=\frac{1}{\Delta x}\sum_{p=1}^{N_p}\omega_p
\int_{t^n}^{t^{n+1}}v_p(t)
S^{m-1}(x_{i+1/2}-x_p(t))\rd{t}\\
&=\sum_{p=1}^{N_p}\omega_p\int_{t^n}^{t^{n+1}} \frac{1}{\Delta x}S^{m-1}(x_{i+1/2}-x_p)\dot{x}_p(t)\rd{t}\\
&=-\sum_{p=1}^{N_p}\omega_p
\int_{t^n}^{t^{n+1}}\frac{\rd{}}{\rd{t}}
\Phi^{m-1}(x_{i+1/2}-x_p(t))\rd{t}\\
&=\sum_{p=1}^{N_p}\omega_p\left[
\Phi^{m-1}(x_{i+1/2}-x_p^n)
-\Phi^{m-1}(x_{i+1/2}-x_p(t^{n+1}))\right].
\end{align}
Notice that $x_p(t^{n+1})
=x_p^n+\Delta t\,v_p^n+\frac{\Delta t^2}{2}E^n(x_p^n)+O(\Delta t^3)=x_p^{n+1}+O(\Delta t^3)$, therefore 
\begin{align}
\left|\Delta t\,\bar J_{i+1/2}-\int_{t^n}^{t^{n+1}}\tilde J_{i+1/2}(t)\rd t\right|
&\le \sum_{p=1}^{N_p} |\omega_p| \left|\Phi^{m-1}(x_{i+1/2}-x_p(t^{n+1}))-\Phi^{m-1}(x_{i+1/2}-x_p^{n+1})\right|\\
&\le C\sum_{p=1}^{N_p} |\omega_p| |x_p(t^{n+1})-x_p^{n+1}| = O(\Delta t)^3, \nonumber
\end{align}
where we used the Lipschitz continuity of $\Phi^{m-1}$ and $C$ is a constant independent of $\Delta t$.

\begin{proposition}
    Scheme VA-I preserves
\begin{itemize}
    \item the total modified energy $\tilde{\mathcal{E}}_T^n= \frac{1}{2}\sum_{p=1}^{N_p} \omega_p (v_p^n)^2 + (r^n)^2-C_0$ for all $n\in \mathbb{N}$;
    \item the Gauss law $E_{i+1/2}^{n}-E_{i-1/2}^{n}=\Delta x (\rho_i^n-1)$ for all $n\in \mathbb{N}$.
\end{itemize}

\end{proposition}

\begin{proof}
For the kinetic energy, we multiply \eqref{VA1_v} by $\sum_{p=1}^{N_p} v_p^{n+1/2}$ to obtain
\begin{align}
    \mathcal{E}_K^{n+1}-\mathcal{E}_K^n&=\frac{1}{2}\sum_{p=1}^{N_p} \omega_p((v_p^{n+1})^2-(v_p^n)^2)=\sum_{i=0}^{N_x-1}\sum_{p=1}^{N_p}\omega_p v_p^{n+1/2} \Delta t E_{i+1/2}^{n+1/2} S^m(x_p^{n+1/2}-x_{i+1/2}).
\end{align}
For the  modified electric energy, we multiply \eqref{VA1_r} by $(r^n+r^{n+1})$ to obtain
\begin{align}
    \tilde{\mathcal{E}}_E^{n+1}-\tilde{\mathcal{E}}_E^n&=(r^{n+1})^2-(r^n)^2=-\sum_{i=0}^{N_x-1}\sum_{p=1}^{N_p}\omega_p v_p^{n+1/2} \Delta t E_{i+1/2}^{n+1/2} S^m(x_p^{n+1/2}-x_{i+1/2}).
\end{align}
Add both, we see that the total modified energy $\tilde{\mathcal{E}}_T^{n}=\mathcal{E}_K^n+\tilde{\mathcal{E}}_E^n$ is preserved.

To show the charge conservation, we assume that $\rho_i(t), x_p(t)$ are time-continuous functions satisfying $$\rho_i(t^n)=\rho_i^n,\quad \rho_i(t^{n+1})=\rho_i^{n+1},\quad x_p(t^n) = x_p^n, \quad x_p(t^{n+1})=x_p^{n+1},$$ 
then we have
\begin{align}
    \label{Charge1}
  \frac{\rho_i^{n+1}-\rho_i^{n}}{ \Delta t} &= \frac{1}{ \Delta t}\int_{t^n}^{t^{n+1}}\frac{\rd{\rho_i(t)}}{\rd{t}}\rd{t} \nonumber \\
 &=\frac{1}{\Delta x \Delta t}\int_{t^n}^{t^{n+1}}\sum_{p=1}^{N_p} \omega_p (-\dot{x}_p(t)) (S^m)'(x_i-x_p(t)) \rd{t} \nonumber\\
 &=\frac{1}{(\Delta x)^2\Delta t}\int_{t^n}^{t^{n+1}}\sum_{p=1}^{N_p} \omega_p \dot{x}_p(t)\Big(S^{m-1}(x_{i-1/2}-x_p(t))-S^{m-1}(x_{i+1/2}-x_p(t))\Big) \rd{t} \nonumber\\
&=\frac{1}{(\Delta x)^2\Delta t} \sum_{p=1}^{N_p} \int_{x^n_p}^{x^{n+1}_p} \omega_p \Big(S^{m-1}(x_{i-1/2}-x_p)-S^{m-1}(x_{i+1/2}-x_p)\Big)\rd{x_p}\nonumber\\
&=\nonumber\frac{-1}{(\Delta x)^2\Delta t} \sum_{p=1}^{N_p} \omega_p \left(\int_{x_{i-1/2}-x^n_p}^{x_{i-1/2}-x^{n+1}_p}S^{m-1}(y)\rd{y}-\int_{x_{i+1/2}-x^n_p}^{x_{i+1/2}-x^{n+1}_p}S^{m-1}(y)\rd{y}\right) \nonumber\\
    &=\frac{1}{\Delta x \Delta t}\sum_{p=1}^{N_p} \omega_p [\Phi^{m-1}(x_{i-1/2}-x_p^n)-\Phi^{m-1}(x_{i-1/2}-x_p^{n+1})\nonumber\\& \quad -\Phi^{m-1}(x_{i+1/2}-x_p^n)+\Phi^{m-1}(x_{i+1/2}-x_p^{n+1})]\nonumber\\
&=\nonumber\frac{\bar{J}_{i-1/2}-\bar{J}_{i+1/2}}{\Delta x}=\frac{1}{\Delta x}\left(\frac{E^{n+1}_{i+1/2}-E^{n+1}_{i-1/2}}{\Delta t}-\frac{E^{n}_{i+1/2}-E^{n}_{i-1/2}}{\Delta t}\right),
\end{align}
where we used \eqref{SmSm-1} in the third line and \eqref{Sm-1_prim} in the second last line above. Therefore, if the Gauss law $E_{i+1/2}^{n}-E_{i-1/2}^{n}=\Delta x (\rho_i^n-1)$ is satisfied at time $t^n$, then it is also satisfied at $t^{n+1}$, which means the charge is preserved.
\end{proof}

\subsection{Scheme VA-II}
\label{subsec: scheme VA-II}

We now introduce a different SAV scheme for the VA system, VA-II, which uses a scaled electric field $g(t,x):=\frac{E(t,x)}{r(t)}$ to first change the system \eqref{VA_SAV} to 
\begin{subequations}
\label{VA_SAV_rg}
\begin{align}
&\partial_t f + v\partial_x f + rg \partial_v f = 0, \\
\label{VA_SAV_rg_g}
&\partial_t (rg) = -J \longrightarrow r \partial_t g=-J+\frac{1}{2}g\int_0^Lg J\rd{x}, \\
\label{VA_SAV_rg_r}
&\frac{\rd{r}}{\rd{t}} = - \frac{1}{2}\int_0^L gJ\rd{x} .
\end{align}
\end{subequations}
A particle discretization for this system reads:
\begin{subequations}
    \label{VA_sd_rg}
\begin{align}
    \dot{x}_p(t) &= v_p(t), \\
    \dot{v}_p(t) &=r(t)g(x_p(t)), \\
    r(t)\dot{g}_{i+1/2}(t)&=-J_{i+1/2}(t)+\frac{1}{2}g_{i+1/2}(t)\left(\sum_{j=0}^{N_x-1} g_{j+1/2}(t) J_{j+1/2}(t) \Delta x\right),\\
    \dot{r}(t) &= - \frac{1}{2}\sum_{i=0}^{N_x-1} g_{i+1/2}(t) J_{i+1/2}(t) \Delta x,
\end{align}
\end{subequations}
where $g_{i+1/2}$ is defined analogously to $J_{i+1/2}$ in \eqref{def:rho_J}, and $g(x_p)$ is defined analogously to $E(x_p)$ in \eqref{E_part}.
To obtain a second-order time discretization, we introduce a time splitting scheme for \eqref{VA_sd_rg}, which consists of the following three subsystems:
\begin{subequations}
\label{VA_splitting}
    \begin{align}
        \varphi_{\Delta t}^{[1]}&\begin{cases}
        \label{VA_sub1}
            \dot{x}_p(t)=v_p(t),\\
             \dot{v}_p(t)=0,\\
            r(t)\dot{g}_{i+1/2}(t)= -J_{i+1/2}(t),\\
            \dot{r}(t)=0,
        \end{cases}\\
        \label{VA_sub2}
        \varphi_{\Delta t}^{[2]} &\begin{cases}
        \dot{x}_p(t)=0,\\
        \dot{v}_p(t)=0,\\
        r(t) \dot{g}_{i+1/2} (t)=\frac{1}{2}g_{i+1/2}(t)\left(\sum_{j=0}^{N_x-1} g_{j+1/2}(t) J_{j+1/2}(t) \Delta x\right),\\
        \dot{r}(t)=0,
        \end{cases}\\
        \varphi_{\Delta t}^{[3]}&\begin{cases}
        \label{VA_sub3}
            \dot{x}_p(t)=0,\\
            \dot{v}_p(t)=r(t)g(x_p(t)),\\
            \dot{g}_{i+1/2}(t)=0,\\
         \dot{r}(t)= - \frac{1}{2}\sum_{i=0}^{N_x-1} g_{i+1/2}(t)J_{i+1/2}(t) \Delta x.
        \end{cases}
    \end{align}
\end{subequations}

We now describe a second-order scheme for each subsystem. We use the superscript $(-)$ to denote the input to a subsystem and the superscript $(+)$ to denote its output.

\subsubsection{First subsystem: $\varphi_{\Delta t}^{[1]}$}
We discretize \eqref{VA_sub1} as
\begin{subequations}
\label{VA-II_sub1}
\begin{align}
    x_p^+-x_p^- & = \Delta t v_p^{-}, \\
    r^- (g_{i+1/2}^+-g_{i+1/2}^- )&  = - \Delta t \bar{J}_{i+1/2}, 
\end{align}
\end{subequations}
where
\begin{align}
\label{VA-II-J}
    \bar{J}_{i+1/2}=\frac{1}{\Delta t}\sum_{p=1}^{N_p} \omega_p [\Phi^{m-1}(x_{i+1/2}-x_p^-)-\Phi^{m-1}(x_{i+1/2}-x_p^+)].
\end{align}

\subsubsection{Second subsystem: $\varphi_{\Delta t}^{[2]}$}

In this subsystem, all components of $g$ have the same scalar growth rate, and we can solve it exactly. Set
\begin{equation}
    g_{i+1/2}(t)=\alpha(t)g_{i+1/2}^-,
    \qquad \alpha(0)=1.
\end{equation}
It follows that
\begin{equation}
\label{alpha_g}
    \sum_{j=0}^{N_x-1} g_{j+1/2}(t)J_{j+1/2}^-\Delta x=\alpha(t)\sum_{i=0}^{N_x-1} g_{i+1/2}^-J_{i+1/2}^-\Delta x,
\end{equation}
where
\begin{equation}
J_{i+1/2}^-=\frac{1}{\Delta x} \sum_{p=1}^{N_p} \omega_p v_p^- S^m (x_{i+1/2} - x_p^-).
\end{equation}
Substituting \eqref{alpha_g} into \eqref{VA_sub2} gives
\begin{equation}
    \dot{\alpha}(t)=\frac{\sum_{j=0}^{N_x-1} g_{j+1/2}^-J_{j+1/2}^-\Delta x}{2r^-}\alpha^2(t),
\end{equation}
whose solution is
\begin{equation}
    \frac{1}{\alpha(t)}=1-\frac{t}{2r^-}\sum_{j=0}^{N_x-1}g_{j+1/2}^-J_{j+1/2}^-\Delta x,
\end{equation}
hence 
\begin{equation}
    g_{i+1/2}^+
    =\frac{g_{i+1/2}^-}
    {1-\dfrac{\Delta t}{2r^-}\sum_{j=0}^{N_x-1}g_{j+1/2}^-J_{j+1/2}^-\Delta x}.
    \label{VA_scaling_update}
\end{equation}

\subsubsection{Third subsystem: $\varphi_{\Delta t}^{[3]}$}
\label{sub:233}

We discretize \eqref{VA_sub3} as
\begin{subequations}
\begin{align}
\label{eq:vplus}
  v_p^+-v_p^- &= \Delta t r^\text{mid} g_p^-,\\
  \label{eq:rplus}
  r^+-r^- &= -\frac{\Delta t}{2}\sum_{i=0}^{N_x-1} g_{i+1/2}^-J_{i+1/2}^\text{mid}\Delta x\left(=-\frac{\Delta t}{2}\sum_{p=1}^{N_p} \omega_p v_p^{\text{mid}}g_p^-\right),
\end{align}
\end{subequations}
where
\begin{align}
r^\text{mid}&=\frac{r^++r^-}{2},\quad  J_{i+1/2}^\text{mid} = \frac{1}{\Delta x} \sum_{p=1}^{N_p} \omega_p v_p^\text{mid} S^m (x_{i+1/2} - x_p^-), \\
 g_p^-&=\sum_{i=0}^{N_x-1} g_{i+1/2}^- S^m(x_p^--x_{i+1/2}), \quad v_p^\text{mid}=\frac{v_p^++v_p^-}{2}.
\end{align}
Note that $r^{\text{mid}}$ can be solved using
\begin{equation}
r^{\text{mid}}=\frac{r^--\frac{\Delta t}{4}\sum_{p=1}^{N_p} \omega_p g_p^- v_p^-}{1+\frac{\Delta t^2}{8}\sum_{p=1}^{N_p} \omega_p (g_p^-)^2}.
\end{equation}

\subsubsection{Full scheme}
Denote $U(t) = (x_p, v_p, r, g_{i+1/2})$ for $p=1, \dots, N_p$ and $i=0, \dots, N_x-1$, and let $U^n\approx U(t_n)$. A symmetric second-order composition is given by
\begin{equation}
    U^{n+1}=\varphi_{\Delta t/2}^{[1]}
    \circ\varphi_{\Delta t/2}^{[2]}
    \circ\varphi_{\Delta t}^{[3]}
    \circ\varphi_{\Delta t/2}^{[2]}
    \circ\varphi_{\Delta t/2}^{[1]}(U^n).
    \label{VA_Strang}
\end{equation}

Note that in this scheme, the electric field $E$ is updated alongside, but independent of the scheme, as
\begin{align}
\label{VA-II_current}
    E_{i+1/2}^{n+1} - E_{i+1/2}^{n} = - \Delta t \bar{J}_{i+1/2}, 
\end{align}
where 
\begin{align}
    \bar{J}_{i+1/2}=\frac{1}{\Delta t}\sum_{p=1}^{N_p} \omega_p [\Phi^{m-1}(x_{i+1/2}-x_p^n)-\Phi^{m-1}(x_{i+1/2}-x_p^{n+1})].
\end{align}

\begin{proposition}
    Scheme VA-II preserves
    \begin{itemize}
    \item the total modified energy $\tilde{\mathcal{E}}_T^n= \frac{1}{2}\sum_{p=1}^{N_p} \omega_p (v_p^n)^2 + (r^n)^2-C_0$ for all $n\in \mathbb{N}$;
    \item the Gauss law $E_{i+1/2}^{n}-E_{i-1/2}^{n}=\Delta x (\rho_i^n-1)$ for all $n\in \mathbb{N}$.
\end{itemize}
\end{proposition}

\begin{proof}

For the energy conservation, since $v_p$ and $r$ are not updated in $\varphi_{\Delta t}^{[1]}$ and $\varphi_{\Delta t}^{[2]}$, we only need to check if $\varphi_{\Delta t}^{[3]}$ preserves the total modified energy. Indeed, multiplying \eqref{eq:vplus} by $v_p^{\text{mid}}$ and multiplying \eqref{eq:rplus} by $r^{\text{mid}}$ yield, respectively,
\begin{align}
    \mathcal{E}_K^+-\mathcal{E}_K^-&=\frac{1}{2}\sum_{p=1}^{N_p} \omega_p((v_p^{+})^2-(v_p^-)^2)=\sum_{i=0}^{N_x-1}\sum_{p=1}^{N_p}\omega_p v_p^\text{mid}\Delta t r^\text{mid} g_{i+1/2}^- S^m(x_p^--x_{i+1/2}),\\
        \tilde{\mathcal{E}}_E^+-\tilde{\mathcal{E}}_E^-&=(r^{+})^2-(r^-)^2=-\sum_{i=0}^{N_x-1}\sum_{p=1}^{N_p}\omega_p v_p^\text{mid}\Delta t r^\text{mid} g_{i+1/2}^- S^m(x_p^--x_{i+1/2}).
\end{align}
Adding the two equations gives $\tilde{\mathcal{E}}_T^+ = \tilde{\mathcal{E}}_T^-$, so the total modified energy is preserved.

Since the scheme VA-II shares the same $E$ update as the scheme VA-I, charge conservation can be proved exactly the same as before.
\end{proof}

\begin{remark}
The two schemes, VA-I and VA-II, introduced above differ significantly. VA-I updates the original variables $(x_p,v_p,E_{i+1/2},r)$, but the update of $r$ is essentially decoupled from the rest of the scheme, whereas VA-II updates the transformed  variables $(x_p,v_p,g_{i+1/2}),r$, but the update of $E$ is decoupled from the main loop. 
\end{remark}

\begin{remark}
Other time integrators can be used for the approximation of the semi-discretized system 
\begin{align*}
\begin{cases}
         \dot{x}_p(t) = v_p(t),\\
         \dot{v}_p(t) = E(t, x_p(t)) = r(t) g(t, x_p(t)), \\
         \dot{E}_{i+1/2}(t) = -{J}_{i+1/2}(t),\\
         \dot{r}(t) = -\frac{1}{2} \sum_{i=0}^{N_x} g_{i+1/2}(t) J_{i+1/2}(t) \Delta x. 
\end{cases}
\end{align*}
Indeed, to ensure the preservation of both charge and modified energy, the following splitting can be used  
 \begin{align}
 \label{va0_1}
     \varphi^{[1]}_{t} \quad &\begin{cases}
         \dot{x}_p(t) = v_p(t),\\
         \dot{E}_{i+1/2}(t) = -{J}_{i+1/2}(t),
     \end{cases}\\
     \label{va0_2}
     \varphi^{[2]}_t \quad &\begin{cases}
         \dot{v}_p(t) = r g_p(t),\\
         \displaystyle \dot{r}(t)
         =-\frac{1}{2}\sum_i g_{i+1/2}(t)J_{i+1/2}(t)\Delta x.  
     \end{cases}
 \end{align}
 To get second order accuracy in time, we use a Strang splitting 
 $\varphi^{[1]}_{\Delta t/2}\circ \varphi^{[2]}_{\Delta t}\circ \varphi^{[1]}_{\Delta t/2}$ combined with a midpoint scheme for \eqref{va0_2}. More precisely, for $\varphi^{[1]}_{\Delta t/2}$, we use the charge preserving technique 
 \begin{eqnarray*}
 x^+_p &=& x^-_p + \frac{\Delta t}{2} v_p^-, \;\;\;\; 
 {E}_{i+1/2}^+ = {E}_{i+1/2}^- -\frac{\Delta t}{2} \bar{J}_{i+1/2}^{\textrm{mid}},\nonumber\\
 \bar{J}_{i+1/2}^{\textrm{mid}} &=& \frac{1}{\Delta t}\sum_{p=1}^{N_p} \omega_p \Big[ \Phi^{m-1}(x_{i+1/2}-x_p^-) - \Phi^{m-1}(x_{i+1/2}-x_p^+) \Big]\nonumber. 
 \end{eqnarray*}
As a second order approximation of $g_{i+1/2}(t)$ (and $g_p(t)$) in $\varphi^{[2]}_t$, we compute $g^\star_{i+1/2}$ which denotes any first order approximation of $g_{i+1/2}(t^{n+1/2}) = E_{i+1/2}(t^{n+1/2})/r(t^{n+1/2})$, for example 
\begin{itemize}
    \item $g^\star_{i+1/2} = \frac{E^{+}_{i+1/2}}{r^- -\frac{\Delta t}{4}\sum_{i=0}^{N_x-1} g^-_{i+1/2}J^-_{i+1/2}\Delta x}$, 
        \item $g^\star_{i+1/2}= \frac{3}{2}g^n_{i+1/2} - \frac{1}{2} g^{n-1}_{i+1/2}$,  in the spirit of  \cite{sav_schro, sav_xavier}, 
    \item $g^\star_{i+1/2} = g^-_{i+1/2} + \frac{\Delta t}{2}\dot{g}^-_{i+1/2}  = g^-_{i+1/2} + \frac{\Delta t}{2}\Big( -\frac{J^-_{i+1/2}}{r^-_{i+1/2}} +\frac{g^-_{i+1/2}}{2r^-_{i+1/2}}\sum_{i=0}^{N_x-1} g^-_{i+1/2} J^-_{i+1/2} \Delta x\Big)$, using \eqref{VA_SAV_rg_g}. 
\end{itemize}  
Once $g^\star_{i+1/2}$ is computed,  
$\varphi^{[2]}_{\Delta t}$ is updated using the midpoint method to ensure modified energy preservation 
 \begin{eqnarray*}
  v_p^{+} &=& v_p^- + \Delta t r^{\textrm{mid}} g^\star_p, \;\;\;\; 
 r^{+} = r^- - \frac{\Delta t}{2}\sum_{p=1}^{N_p}\sum_{i=0}^{N_x-1} \omega_p g_{p}^\star v_p^{\textrm{mid}} S^m(x_{i+1/2}-x_p^-)\Delta x,\nonumber 
 \end{eqnarray*}
with $r^{\textrm{mid}} = (r^-+r^{+})/2$ and $v_p^{\textrm{mid}} = (v_p^-+v_p^{+})/2$,  
 which is solved as in Subsection \ref{sub:233}. 

For this second order time integrator, one can also prove charge preservation thanks to the current $\bar{J}_{i+1/2}$ in $\varphi^{[1]}_{t}$ and preservation of the modified total energy thanks to the midpoint method in $\varphi^{[2]}_{t}$, as for VA-I and VA-II. Of course, other time discretization may be derived from the SAV semi-discretization. We restricted ourselves to the presentation of VA-I and VA-II within this manuscript. 
\end{remark}

\section{Electromagnetic Vlasov-Maxwell system}

In this section, we consider the VM system in 1D2V,
\begin{subequations}
\begin{align}
    \label{VM_f}
    \partial_t f + v_x\partial_x f +   (E_x + Bv_y)\partial_{v_x} f & + (E_y - Bv_x)\partial_{v_y} f = 0,\\
    \label{VM_Ex}
    \partial_t E_x & = -J_x, \quad J_x = \int_{\mathbb{R}^2} v_x f \rd{\boldsymbol{v}}, \\
    \label{VM_Ey}
    \partial_t E_y & = -J_y - \partial_x B, \quad J_y = \int_{\mathbb{R}^2} v_y f \rd{\boldsymbol{v}}, \\
    \label{VM_B}
    \partial_t B & = -\partial_x E_y, 
\end{align}
\end{subequations}
where $f(t,x,v_x,v_y)$ is the distribution function with $x\in [0, L]$ and $\boldsymbol{v}= (v_x, v_y) \in \mathbb{R}^2$. We again assume periodic boundary condition in $x$ and vanishing boundary in $\boldsymbol{v}$. The transverse fields $E_y(0,x)$ and $B(0,x)$ are set to zero unless specified otherwise. The electric field $E_x(0,x)$ is constrained by Gauss's law initially. \eqref{VM_f} and \eqref{VM_Ex} then propagate Gauss's law for all later times. As in the VA system, it is challenging to preserve both the total energy and the charge in the discretized VM system, and we introduce the SAV reformulation to achieve this goal.

\subsection{SAV formulation and the particle method}

As in \cite{crouseilles2026}, we introduce two SAVs $r_{x,y}^2(t)=\frac{1}{2}\int_0^L E_{x,y}^2(t, x) \rd{x}+C_0$, where $C_0$ is a constant chosen to ensure that the right-hand side is always positive. The VM system is then reformulated as
\begin{subequations}
    \label{VM_SAV}
\begin{align}
    \label{VM_SAV_f}
    \partial_t f &+ v_x\partial_x f +   (E_x + Bv_y)\partial_{v_x} f  + (E_y - Bv_x)\partial_{v_y} f = 0,\\
    \label{VM_SAV_Ex}
    \partial_t E_x & = -J_x, \\
    \label{VM_SAV_Ey}
    \partial_t E_y & = -J_y - \partial_x B, \\
    \label{VM_SAV_B}
    \partial_t B & = -\partial_x E_y, \\
    \label{VM_SAV_rx}
    \dot{r}_x & = - \frac{1}{2r_x} \int_0^L E_x J_x \rd{x}, \\
    \label{VM_SAV_ry}
    \dot{r}_y & = - \frac{1}{2r_y} \int_0^L E_y (J_y + \partial_x B) \rd{x}.
\end{align}
\end{subequations}

Similarly to the VA system, we use the particle approximation
\begin{equation}
    f(t,x,v_x,v_y)\approx \sum_{p=1}^{N_p} \omega_p \delta(x-x_p(t))\delta(v_x-v_{x,p}(t))\delta(v_y-v_{y,p}(t)),
\end{equation}
and divide the spatial domain into $N_x$ uniform cells of size $\Delta x = L/N_x$. On a Yee's lattice  \cite{yee}, $E_x, J_x, B$ are defined at the cell centers $x_{i+1/2}=(i+1/2)\Delta x$, while $E_y, J_y$, $\rho$ 
are defined at the cell interfaces $x_i=i\Delta x$, $i=0,\dots,N_x-1$. $E_x(0,x)$ is initialized by the same discrete Gauss-law/Poisson construction as in \eqref{VA_Poisson}, using the charge density deposited from the particles as in \eqref{def:rho_J}. The system \eqref{VM_SAV} is then approximated using particles as
\begin{subequations}
\label{VM_semidiscrete}    
\begin{align}
    \label{VM_sd_x}
    \dot{x}_p & = v_{x,p}, \\
    \label{VM_sd_vx}
    \dot{v}_{x,p} & =E_{x,p} + B_p v_{y,p}, \\
    \label{VM_sd_vy}
    \dot{v}_{y,p} & =E_{y,p} - B_p v_{x,p}, \\
    \label{VM_sd_Ex}
    \dot{E}_{x,i+1/2} & = -J_{x,i+1/2}, \\
    \label{VM_sd_Ey}
    \dot{E}_{y,i} & = -J_{y,i} - \frac{B_{i+1/2}-B_{i-1/2}}{\Delta x}, \\
    \label{VM_sd_B}
    \dot{B}_{i+1/2} & = -\frac{E_{y,i+1}-E_{y,i}}{\Delta x}, \\
    \label{VM_sd_rx}
    \dot{r}_x & = - \frac{1}{2r_x} \sum_{i=0}^{N_x-1} E_{x,i+1/2} J_{x,i+1/2} \Delta x, \\
    \label{VM_sd_ry}
    \dot{r}_y & = - \frac{1}{2r_y} \sum_{i=0}^{N_x-1} E_{y,i} \left(J_{y,i} + \frac{B_{i+1/2}-B_{i-1/2}}{\Delta x}\right) \Delta x.
\end{align}
\end{subequations}
We use the same B-spline shape functions for the particle-mesh coupling as in Subsection \ref{sec:SAV_VA}, specifically,
\begin{align}
    J_{x,i+1/2}&=\frac{1}{\Delta x}\sum_{p=1}^{N_p}\omega_pv_{x,p}
    S^m(x_{i+1/2}-x_p),&
    J_{y,i}&=\frac{1}{\Delta x}\sum_{p=1}^{N_p}\omega_pv_{y,p}
    S^m(x_i-x_p),\\
    E_{x,p}&=\sum_{i=0}^{N_x-1}E_{x,i+1/2}S^m(x_p-x_{i+1/2}),&
    E_{y,p}&=\sum_{i=0}^{N_x-1}E_{y,i}S^m(x_p-x_i),\\
    B_p&=\sum_{i=0}^{N_x-1}B_{i+1/2}S^m(x_p-x_{i+1/2}).
\end{align}

\subsection{Scheme VM-I}
\label{subsec:VM-I}
Similarly to VA-I, we can use predictors to obtain a second-order scheme, VM-I:
\begin{subequations}
\begin{align}
    v_{x,p}^* - v_{x,p}^n& =  \frac{\Delta t}{2} (E_{x,p}^n + v_{y,p}^* B_p^n), \\
    v_{y,p}^* - v_{y,p}^n& = \frac{\Delta t}{2} (E_{y,p}^n - v_{x,p}^* B_p^n), \\
    x^{n+1}_p-x^n_p& = \Delta t v_{x,p}^*, \\
    E_{x,i+1/2}^{n+1}-E_{x,i+1/2}^n& = - \Delta t \bar{J}_{x,i+1/2}^{n+1/2}, \\
    \label{VM1_Ey}
    E_{y,i}^{n+1}-E_{y,i}^n& =- \Delta t J_{y,i}^{*,n+1/2} - \frac{B^{n+1/2}_{i+1/2}-B^{n+1/2}_{i-1/2}}{\Delta x} \Delta t, \\
    \label{VM1_Bz}
    B^{n+1}_{i+1/2}-B^{n}_{i+1/2}   & = -\Delta t \frac{E_{y,i+1}^{n+1/2}-E_{y,i}^{n+1/2}}{\Delta x}, \\
    {v}^{n+1}_{x,p}-{v}^{n}_{x,p}   & =\Delta t E_{x,p}^{n+1/2}+ \Delta t  B_{p}^{n+1/2}v^{n+1/2}_{y,p}, \\
    {v}^{n+1}_{y,p}-{v}^{n}_{y,p}   & = \Delta t E_{y,p}^{n+1/2} - \Delta t B_{p}^{n+1/2}v^{n+1/2}_{x,p}, \\
    r_x^{n+1}-r_x^{n} & = - \frac{\Delta t}{2 r_x^{n+1/2}} \sum_{i=0}^{N_x-1}\sum_{p=1}^{N_p} \omega_p E_{x,i+1/2}^{n+1/2} v_{x,p}^{n+1/2} S^m(x_{i+1/2}-x_p^{n+1/2}), \\
    \label{VM1_ry}
    r_y^{n+1}-r_y^{n} & = - \frac{\Delta t}{2 r_y^{n+1/2}} \sum_{i=0}^{N_x-1}\sum_{p=1}^{N_p} \omega_p E_{y,i}^{n+1/2} v_{y,p}^{n+1/2} S^m(x_i - x_p^{n+1/2}) \nonumber\\
    &\quad - \frac{\Delta t}{2 r_y^{n+1/2}} \sum_{i=0}^{N_x-1} E_{y,i}^{n+1/2} (B^{n+1/2}_{i+1/2}-B^{n+1/2}_{i-1/2}),
\end{align}
\end{subequations}
where
\begin{align}
    \bar{J}_{x,i+1/2}^{n+1/2} & =\frac{1}{\Delta t}\sum_{p=1}^{N_p} \omega_p [\Phi^{m-1}(x_{i+1/2}-x_p^n)-\Phi^{m-1}(x_{i+1/2}-x_p^{n+1})], \\
    J_{y,i}^{*,n+1/2} & = \frac{1}{\Delta x} \sum_{p=1}^{N_p} \omega_p v_{y,p}^* S^m(x_{i}-x_p^{n+1/2}), \\
    x_p^{n+1/2}&=\frac{x_p^n+x_p^{n+1}}{2},\quad
    v_{x,p}^{n+1/2}=\frac{v_{x,p}^n+v_{x,p}^{n+1}}{2},\quad v_{y,p}^{n+1/2}=\frac{v_{y,p}^n+v_{y,p}^{n+1}}{2},\\
    r_x^{n+1/2}&=\frac{r_x^n+r_x^{n+1}}{2},\quad r_y^{n+1/2}=\frac{r_y^n+r_y^{n+1}}{2},\quad
    B_{i+1/2}^{n+1/2}=\frac{B_{i+1/2}^n+B_{i+1/2}^{n+1}}{2},\\
    E_{x,i+1/2}^{n+1/2}&=\frac{E_{x,i+1/2}^n+E_{x,i+1/2}^{n+1}}{2},\quad
    E_{y,i}^{n+1/2}=\frac{E_{y,i}^n+E_{y,i}^{n+1}}{2}.
\end{align}
Here $E_{x,p}^n$, $E_{y,p}^n$, and $B_p^n$ are the particle interpolants
defined above evaluated at $x_p^n$. The midpoint particle fields in the
velocity update are
\begin{align}
    E_{x,p}^{n+1/2}
    &=\sum_{i=0}^{N_x-1}E_{x,i+1/2}^{n+1/2}
      S^m(x_p^{n+1/2}-x_{i+1/2}), \quad 
    E_{y,p}^{n+1/2}
    =\sum_{i=0}^{N_x-1}E_{y,i}^{n+1/2}
      S^m(x_p^{n+1/2}-x_i),\\
    B_p^{n+1/2}
    &=\sum_{i=0}^{N_x-1}B_{i+1/2}^{n+1/2}
      S^m(x_p^{n+1/2}-x_{i+1/2}).
\end{align}
Unlike the VA case, the field equations \eqref{VM1_Ey}-\eqref{VM1_Bz} are not completely explicit. Nevertheless, the implicit electromagnetic field equations can be reformulated as a linear system
\begin{align}
    -kE_{y,i-1}^{n+1}+(1+2k)E_{y,i}^{n+1}-kE_{y,i+1}^{n+1}
    &= kE_{y,i-1}^{n} + (1-2k)E_{y,i}^{n} + kE_{y,i+1}^{n} \nonumber\\
    &\quad - \Delta t J_{y,i}^{*,n+1/2} - \frac{\Delta t}{\Delta x}(B^n_{i+1/2}-B^n_{i-1/2}),
\end{align}
where $k=\frac{\Delta t^2}{4\Delta x^2}$. This system can be solved at the cost of an explicit method. The implicitness also ensures the unconditional stability of the field solver.

\begin{proposition}
    Scheme VM-I preserves
    \begin{itemize}
    \item the total modified energy
    \begin{equation*}
    \begin{aligned}
    \tilde{\mathcal{E}}_T^n
    &=\frac{1}{2} \sum_{p=1}^{N_p} \omega_p [(v_{x,p}^n)^2 + (v_{y,p}^n)^2]
    + (r_x^n)^2 + (r_y^n)^2 - 2C_0  + \frac{1}{2} \sum_{i=0}^{N_x-1} (B_{i+1/2}^n)^2\Delta x
    \end{aligned}
    \end{equation*}
    for all $n\in \mathbb{N}$;
    \item the Gauss law $E_{x,i+1/2}^{n}-E_{x,i-1/2}^{n}=\Delta x (\rho_i^n-1)$ for all $n\in \mathbb{N}$.
    \end{itemize}
\end{proposition}

\begin{proof}
    We have
    \begin{align}
    \frac{1}{2}\sum_{p=1}^{N_p} \omega_p \left[ (v_{x,p}^{n+1})^2 - (v_{x,p}^n)^2  \right]
    &= \sum_{p=1}^{N_p} \omega_p v_{x,p}^{n+1/2} \Delta t E_x^{n+1/2}(x_p^{n+1/2}) \nonumber\\
    &\quad +\sum_{p=1}^{N_p} \omega_p v_{x,p}^{n+1/2} v_{y,p}^{n+1/2} \Delta t B^{n+1/2}(x_p^{n+1/2}),\\
    \frac{1}{2}\sum_{p=1}^{N_p} \omega_p \left[ (v_{y,p}^{n+1})^2 - (v_{y,p}^n)^2  \right]
    &= \sum_{p=1}^{N_p} \omega_p v_{y,p}^{n+1/2} \Delta t E_y^{n+1/2}(x_p^{n+1/2}) \nonumber\\
    &\quad - \sum_{p=1}^{N_p} \omega_p v_{x,p}^{n+1/2} v_{y,p}^{n+1/2} \Delta t B^{n+1/2}(x_p^{n+1/2}),\\
    (r_x^{n+1})^2 - (r_x^n)^2  & = -\Delta t \sum_{p=1}^{N_p} \omega_p E_{x}^{n+1/2} (x_p^{n+1/2}) v_{x,p}^{n+1/2},  \\
    (r_y^{n+1})^2 - (r_y^n)^2  & = -\Delta t \sum_{p=1}^{N_p} \omega_p E_{y}^{n+1/2} (x_p^{n+1/2}) v_{y,p}^{n+1/2} \nonumber\\
    &\quad - \Delta t \sum_{i=0}^{N_x-1} E_{y,i}^{n+1/2} (B^{n+1/2}_{i+1/2}-B^{n+1/2}_{i-1/2}), \\
    \frac{1}{2}\sum_{i=0}^{N_x-1} \left( (B^{n+1}_{i+1/2})^2 - (B^{n}_{i+1/2})^2 \right) \Delta x
    & = -\Delta t  \sum_{i=0}^{N_x-1} B^{n+1/2}_{i+1/2} (E_{y,i+1}^{n+1/2}-E_{y,i}^{n+1/2}) \nonumber\\
    &=\Delta t \sum_{i=0}^{N_x-1} E_{y,i}^{n+1/2} (B^{n+1/2}_{i+1/2}-B^{n+1/2}_{i-1/2}).
\end{align}
Therefore, $\tilde{\mathcal{E}}^{n+1}_T=\tilde{\mathcal{E}}^n_T$.

The proof of the Gauss law preservation is identical to the VA system, since it is not affected by $v_y$, $E_y$ and $B$.
\end{proof}
\subsection{Scheme VM-II}
\label{subsec:VM-II}
Introducing the scaled electric field $g_{x,y}(t, x) := \frac{E_{x,y}(t, x)}{r_{x,y}(t)}$, we rewrite \eqref{VM_semidiscrete} into 
\begin{subequations}
\label{VM_sd_rg}
\begin{align}
    \dot{x}_p(t) &= v_{x,p}(t), \\
    \dot{v}_{x,p}(t) &= r_x(t) g_{x,p}(t)+B_p(t)v_{y,p}(t), \\
    \dot{v}_{y,p}(t) &= r_y(t) g_{y,p}(t)-B_p(t)v_{x,p}(t), \\
    r_x(t)\dot{g}_{x,i+1/2}(t) &= -J_{x,i+1/2}(t)+\frac{1}{2} g_{x,i+1/2}(t) \left(\sum_{j=0}^{N_x-1} g_{x,j+1/2}(t) J_{x,j+1/2}(t) \Delta x\right), \\
    r_y(t)\dot{g}_{y,i}(t) &= -J_{y,i}(t)-\frac{B_{i+1/2}(t)-B_{i-1/2}(t)}{\Delta x}+\frac{1}{2}g_{y,i}(t) \left(\sum_{j=0}^{N_x-1} g_{y,j}(t) J_{y,j}(t) \Delta x \right) \nonumber\\ 
    &\quad+\frac{1}{2}g_{y,i}(t) \left(\sum_{j=0}^{N_x-1} g_{y,j}(t) (B_{j+1/2}(t)-B_{j-1/2}(t))\right), \\
    \dot{r}_x(t) &= - \frac{1}{2} \sum_{i=0}^{N_x-1} g_{x,i+1/2}(t) J_{x,i+1/2}(t) \Delta x, \\
    \dot{r}_y(t) &= - \frac{1}{2} \sum_{i=0}^{N_x-1} g_{y,i}(t) J_{y,i}(t) \Delta x -\frac{1}{2} \sum_{i=0}^{N_x-1} g_{y,i}(t) (B_{i+1/2}(t)-B_{i-1/2}(t)),\\
    \dot{B}_{i+1/2}(t)&=-\frac{r_y(t)g_{y,i+1}(t)-r_y(t)g_{y,i}(t)}{\Delta x}.
\end{align}
\end{subequations}
Here the scaled fields at particle positions are defined as
\begin{align}
    g_{x,p}(t)=\sum_{i=0}^{N_x-1}g_{x,i+1/2}(t)
    S^m(x_p(t)-x_{i+1/2}), \quad
    g_{y,p}(t)&=\sum_{i=0}^{N_x-1}g_{y,i}(t)
    S^m(x_p(t)-x_i).
\end{align}

We introduce a time splitting for \eqref{VM_sd_rg}, similar to VA-II, consisting of four subflows $\varphi_{\Delta t}^{[\ell]}, \ell=1, 2, 3,4$:
\begin{subequations}
\begin{align}
  \varphi^{[1]}_{\Delta t} \;\;\;   & \begin{cases}
\dot{x}_p(t) = v_{x,p}(t), \\
r_x(t)\dot{g}_{x,i+1/2}(t) =-J_{x,i+1/2}(t),\\
r_y(t)\dot{g}_{y,i}(t)=-J_{y,i}(t),
\end{cases}\\
  \varphi^{[2]}_{\Delta t} \;\;\;   & \begin{cases}
  r_x \dot{g}_{x,i+1/2} = \frac{1}{2}g_{x,i+1/2}(t)\left(\sum_{j=0}^{N_x-1} g_{x,j+1/2}(t) J_{x,j+1/2}(t) \Delta x\right),\\
  r_y \dot{g}_{y,i} = \frac{1}{2}g_{y,i}(t)\left(\sum_{j=0}^{N_x-1} g_{y,j}(t) J_{y,j}(t) \Delta x\right),
  \end{cases}\\
  \varphi^{[3]}_{\Delta t} \;\;\;     & \begin{cases}
\dot{v}_{x,p}(t) = r_x(t) g_{x,p}(t),\\
\dot{v}_{y,p}(t) = r_y(t) g_{y,p}(t),\\
\dot{r}_x(t) = 
- \frac{1}{2} \sum_{i=0}^{N_x-1} g_{x,i+1/2}(t) J_{x,i+1/2}(t) \Delta x, \\
\dot{r}_y(t) = 
- \frac{1}{2} \sum_{i=0}^{N_x-1} g_{y,i}(t) J_{y,i}(t) \Delta x,
\end{cases} \\
  \varphi^{[4]}_{\Delta t} \;\;\;     & \begin{cases}
  \label{VM_sd_phi4}
\dot{v}_{x,p}(t) = B_p(t) v_{y,p}(t),\\
\dot{v}_{y,p}(t) = -B_p(t) v_{x,p}(t),\\
\dot{B}_{i+1/2}(t)=-\frac{r_y(t)g_{y,i+1}(t)-r_y(t)g_{y,i}(t)}{\Delta x},\\
\dot{r}_y(t) = - \frac{1}{2} \sum_{i=0}^{N_x-1} g_{y,i}(t) (B_{i+1/2}(t)-B_{i-1/2}(t)),\\
r_y\dot{g}_{y,i}(t) = -\frac{B_{i+1/2}(t)-B_{i-1/2}(t)}{\Delta x} + \frac{1}{2}g_{y,i}(t) \left(\sum_{j=0}^{N_x-1} g_{y,j}(t) (B_{j+1/2}(t)-B_{j-1/2}(t))\right).
       \end{cases}
\end{align}
\end{subequations}
In each subsystem, the variables that are not written remain constant. We now introduce a second-order scheme for each subsystem. We again use the superscript $(-)$ to denote the input to a subsystem and the superscript $(+)$ to denote its output.

\subsubsection{First subsystem: $\varphi^{[1]}_{\Delta t}$}
The update steps are
\begin{subequations}
\begin{align}
    x_p^+-x_p^-&=\Delta t v_{x,p}^-,\\
    r_x^-(g_{x,i+1/2}^+-g_{x,i+1/2}^-)
    &=-\Delta t \bar{J}_{x,i+1/2},\\
    r_y^-(g_{y,i}^+-g_{y,i}^-)
    &=-\Delta t J^{\text{mid}}_{y,i},
    \label{VM_rg_current_update}
\end{align}
\end{subequations}
where
\begin{align}
    x_p^{\mathrm{mid}}&=\frac{x_p^-+x_p^+}{2},\\
    \bar{J}_{x,i+1/2}&=\frac{1}{\Delta t}\sum_{p=1}^{N_p} \omega_p [\Phi^{m-1}(x_{i+1/2}-x_p^-)-\Phi^{m-1}(x_{i+1/2}-x_p^+)],\\
    J_{y,i}^{\text{mid}}
    &=\frac{1}{\Delta x}\sum_{p=1}^{N_p}\omega_pv_{y,p}^-
      S^m(x_i-x_p^{\mathrm{mid}}).
\end{align}

\subsubsection{Second subsystem: $\varphi_{\Delta t}^{[2]}$}
Similarly to the second subsystem in VA-II, we update $g_x$ and $g_y$ with
\begin{equation}
    g_{x,i+1/2}^+
    =\frac{g_{x,i+1/2}^-}
    {1-\dfrac{\Delta t}{2r_x^-}\sum_{j=0}^{N_x-1}g_{x,j+1/2}^- J_{x,j+1/2}^-\Delta x},\quad
    g_{y,i}^+
    =\frac{g_{y,i}^-}
    {1-\dfrac{\Delta t}{2r_y^-}\sum_{j=0}^{N_x-1}g_{y,j}^-J_{y,j}^-\Delta x}.
\end{equation}

\subsubsection{Third subsystem: $\varphi_{\Delta t}^{[3]}$}
The update steps are
\begin{subequations}
\begin{align}
    v_{x,p}^+-v_{x,p}^-&=\Delta t r_x^{\mathrm{mid}}g_{x,p}^-,&
    v_{y,p}^+-v_{y,p}^-&=\Delta t r_y^{\mathrm{mid}}g_{y,p}^-,\\
    r_x^+-r_x^-&=-\frac{\Delta t}{2}\sum_{p=1}^{N_p}\omega_p
    g_{x,p}^-v_{x,p}^{\mathrm{mid}},&
    r_y^+-r_y^-&=-\frac{\Delta t}{2}\sum_{p=1}^{N_p}\omega_p
    g_{y,p}^-v_{y,p}^{\mathrm{mid}},
    \label{VM_rg_electric_midpoint}
\end{align}
\end{subequations}
where
\begin{align}
    v_{x,p}^{\mathrm{mid}}&=\frac{v_{x,p}^-+v_{x,p}^+}{2},\quad
    v_{y,p}^{\mathrm{mid}}=\frac{v_{y,p}^-+v_{y,p}^+}{2},\\
    r_x^{\mathrm{mid}}&=\frac{r_x^-+r_x^+}{2},\quad
    r_y^{\mathrm{mid}}=\frac{r_y^-+r_y^+}{2}.
\end{align}

Similarly to the third subsystem in VA-II, we calculate 
\begin{align}
    r_x^{\mathrm{mid}}
    =\frac{r_x^--\dfrac{\Delta t}{4}\sum_{p=1}^{N_p}\omega_p
    g_{x,p}^-v_{x,p}^-}
    {1+\dfrac{\Delta t^2}{8}\sum_{p=1}^{N_p}\omega_p(g_{x,p}^-)^2},\quad
    r_y^{\mathrm{mid}}
    =\frac{r_y^--\dfrac{\Delta t}{4}\sum_{p=1}^{N_p}\omega_p
    g_{y,p}^-v_{y,p}^-}
    {1+\dfrac{\Delta t^2}{8}\sum_{p=1}^{N_p}\omega_p(g_{y,p}^-)^2}.
    \label{VM_rg_r_midpoint}
\end{align}

\subsubsection{Fourth subsystem: $\varphi_{\Delta t}^{[4]}$}
Adding up the last two equations in \eqref{VM_sd_phi4} gives 
\begin{align}
    \frac{d}{dt}(r_yg_{y,i})=-\frac{B_{i+1/2}(t)-B_{i-1/2}(t)}{\Delta x}.
\end{align}
Let $G_{y,i}=r_y g_{y,i}$, then the update steps are 
\begin{subequations}
\begin{align}
    \label{VM2_4_vx}
    v_{x,p}^+ - v_{x,p}^- &=\Delta t B^\text{mid}_p v_{y,p}^\text{mid}, \\
    \label{VM2_4_vy}
    v_{y,p}^+ - v_{y,p}^- &=-\Delta tB^\text{mid}_p v_{x,p}^\text{mid},\\
    \label{VM2_4_B}
    B_{i+1/2}^+-B_{i+1/2}^- &= -\frac{\Delta t}{\Delta x} (G_{y,i+1}^\text{mid}-G_{y,i}^\text{mid}),\\
    \label{VM2_4_r}
    r_y^+ - r_y^- &= - \frac{\Delta t}{2r_y^\text{mid}} \sum_{i=0}^{N_x-1} G_{y,i}^\text{mid}(B_{i+1/2}^{\text{mid}}-B_{i-1/2}^{\text{mid}}),\\
    \label{VM2_4_G}
    G_{y,i}^+-G_{y,i}^-&=-\frac{\Delta t}{\Delta x} (B_{i+1/2}^\text{mid}-B_{i-1/2}^\text{mid}),
\end{align}
\end{subequations}
where $G_{y,i}^{\pm}=r_y^{\pm}g_{y,i}^{\pm}$ and
\begin{align}
    v_{x,p}^{\mathrm{mid}}&=\frac{v_{x,p}^-+v_{x,p}^+}{2},\quad
    v_{y,p}^{\mathrm{mid}}=\frac{v_{y,p}^-+v_{y,p}^+}{2},\\
    B_{i+1/2}^{\mathrm{mid}}&=\frac{B_{i+1/2}^-+B_{i+1/2}^+}{2},\quad
    G_{y,i}^{\mathrm{mid}}=\frac{G_{y,i}^-+G_{y,i}^+}{2},\\
    r_y^{\mathrm{mid}}&=\frac{r_y^-+r_y^+}{2},\quad
    B_p^{\mathrm{mid}}=\sum_{i=0}^{N_x-1}B_{i+1/2}^{\mathrm{mid}}
    S^m(x_p^--x_{i+1/2}).
\end{align}
In \eqref{VM2_4_r}, $G_{y,i}^\text{mid}/r_y^\text{mid}$ is used as a
second-order midpoint approximation of $g_{y,i}$.

Similarly to VM-I, we first solve $G_y^+$ and $B^+$ with
\begin{align}
    -kG_{y,i-1}^++(1+2k)G_{y,i}^+-kG_{y,i+1}^+ = kG_{y,i-1}^- + (1-2k)G_{y,i}^- + kG_{y,i+1}^- - \frac{\Delta t}{\Delta x}(B^-_{i+1/2}-B^-_{i-1/2}),
\end{align}
where $k=\frac{\Delta t^2}{4\Delta x^2}$. Afterward, multiplying \eqref{VM2_4_r} by $2r_y^\text{mid}$ gives the $r_y$ update
\begin{align}
    (r_y^+)^2-(r_y^-)^2
    &=-\Delta t\sum_{i=0}^{N_x-1}G_{y,i}^{\text{mid}}
    (B_{i+1/2}^{\text{mid}}-B_{i-1/2}^{\text{mid}})\nonumber\\
    &=\frac{1}{2}\sum_{i=0}^{N_x-1}
    \left[(G_{y,i}^+)^2
    -(G_{y,i}^-)^2\right]\Delta x,
\end{align}
and hence $g_{y,i}^+=G_{y,i}^+/r_y^+$.

Finally, with $\beta_p=\frac{\Delta t}{2}B_p^{\mathrm{mid}}$, the velocity
update is
\begin{align}
    v_{x,p}^+
    &=\frac{1-\beta_p^2}{1+\beta_p^2}v_{x,p}^-
      +\frac{2\beta_p}{1+\beta_p^2}v_{y,p}^-,\\
    v_{y,p}^+
    &=-\frac{2\beta_p}{1+\beta_p^2}v_{x,p}^-
      +\frac{1-\beta_p^2}{1+\beta_p^2}v_{y,p}^-.
\end{align}

\subsubsection{Full scheme}
Denote $U(t) = (x_p,v_{x,p},v_{y,p},B_{i+1/2},r_x,r_y,g_{x,i+1/2},g_{y,i})$ for $p=1, \dots, N_p$ and $i=0, \dots, N_x-1$, and its discretization $U^n\approx U(t_n)$. A symmetric second-order composition is 
\begin{align}
    U^{n+1}
    =\varphi_{\Delta t/2}^{[1]}
    \circ\varphi_{\Delta t/2}^{[2]}
    \circ\varphi_{\Delta t/2}^{[3]}
    \circ\varphi_{\Delta t}^{[4]}
    \circ\varphi_{\Delta t/2}^{[3]}
    \circ\varphi_{\Delta t/2}^{[2]}
    \circ\varphi_{\Delta t/2}^{[1]}(U^n).
    \label{VM_rg_Strang}
\end{align}

The electric field $E_x$ and $E_y$ are updated alongside the system as
\begin{align}
    E_{x,i+1/2}^{n+1}-E_{x,i+1/2}^n& = - \Delta t \bar{J}_{x,i+1/2}^{n+1/2}, \\
    E_{y,i}^{n+1}-E_{y,i}^n& =- \Delta t J_{y,i}^{n+1/2} - \frac{B^{n+1/2}_{i+1/2}-B^{n+1/2}_{i-1/2}}{\Delta x} \Delta t,
\end{align}
where
\begin{align}
    \bar{J}_{x,i+1/2}^{n+1/2} & =\frac{1}{\Delta t}\sum_{p=1}^{N_p} \omega_p [\Phi^{m-1}(x_{i+1/2}-x_p^n)-\Phi^{m-1}(x_{i+1/2}-x_p^{n+1})], \\
    J_{y,i}^{n+1/2} & = \frac{1}{\Delta x} \sum_{p=1}^{N_p} \omega_p v_{y,p}^{n+1/2} S^m(x_{i}-x_p^{n+1/2}), \\
    B^{n+1/2}_{i+1/2} & =\frac{1}{2}(B_{i+1/2}^n+B_{i+1/2}^{n+1}), \\
    v_{y,p}^{n+1/2} & =\frac{1}{2}(v_{y,p}^n+v_{y,p}^{n+1}).
\end{align}

\begin{proposition}
    Scheme VM-II preserves
    \begin{itemize}
    \item the total modified energy
    \begin{equation*}
    \begin{aligned}
    \tilde{\mathcal{E}}_T^n
    &=\frac{1}{2} \sum_{p=1}^{N_p} \omega_p [(v_{x,p}^n)^2 + (v_{y,p}^n)^2]
    + (r_x^n)^2 + (r_y^n)^2 - 2C_0  + \frac{1}{2} \sum_{i=0}^{N_x-1} (B_{i+1/2}^n)^2\Delta x
    \end{aligned}
    \end{equation*}
    for all $n\in \mathbb{N}$;
    \item the Gauss law $E_{x,i+1/2}^{n}-E_{x,i-1/2}^{n}=\Delta x (\rho_i^n-1)$ for all $n\in \mathbb{N}$.
    \end{itemize}
\end{proposition}

\begin{proof} 
Since $\tilde{\mathcal{E}}^n_T$ remains constant along $\varphi^{[1]}_{\Delta t}$ and $\varphi^{[2]}_{\Delta t}$. We only need to check $\varphi^{[3]}_{\Delta t}$ and $\varphi^{[4]}_{\Delta t}$.

In $\varphi^{[3]}_{\Delta t}$, we have
\begin{align*}
    \frac{1}{2}\sum_{p=1}^{N_p} \omega_p \left( (v_{x,p}^+)^2 - (v_{x,p}^-)^2 \right) & = \Delta t r_x^\text{mid} \sum_{p=1}^{N_p} \omega_p v_{x,p}^\text{mid} g_{x,p}^-,\\
    \frac{1}{2}\sum_{p=1}^{N_p} \omega_p \left( (v_{y,p}^+)^2 - (v_{y,p}^-)^2 \right) & = \Delta t r_y^\text{mid} \sum_{p=1}^{N_p} \omega_p v_{y,p}^\text{mid} g_{y,p}^-,\\
    (r_x^+)^2 - (r_x^-)^2 & = -\Delta t r_x^\text{mid} \sum_{p=1}^{N_p} \omega_p v_{x,p}^\text{mid} g_{x,p}^-,\\
    (r_y^+)^2 - (r_y^-)^2 & = -\Delta t r_y^\text{mid} \sum_{p=1}^{N_p} \omega_p v_{y,p}^\text{mid} g_{y,p}^-.
\end{align*}
Summing these equalities leads to $\tilde{\mathcal{E}}_T^+=\tilde{\mathcal{E}}_T^-$ since the electromagnetic field $B$ is unchanged during this step. 

In $\varphi^{[4]}_{\Delta t}$, we multiply \eqref{VM2_4_vx} by $v_{x,p}^\text{mid}$ and \eqref{VM2_4_vy} by $v_{y,p}^\text{mid}$, their sum vanishes, which means the midpoint preserves the kinetic energy. We then multiply \eqref{VM2_4_B} by $B_{i+1/2}^\text{mid}$ and \eqref{VM2_4_r} by $2r_y^\text{mid}$ to get 
\begin{align}
    \frac{1}{2} \sum_{i=0}^{N_x-1} \left( (B^+_{i+1/2})^2 - (B^-_{i+1/2})^2 \right) \Delta x & = -\Delta t \sum_{i=0}^{N_x-1} B^\text{mid}_{i+1/2} (G_{y,i+1}^\text{mid}-G_{y,i}^\text{mid})\nonumber\\&= \Delta t\sum_{i=0}^{N_x-1}G_{y,i}^{\text{mid}}
    (B_{i+1/2}^{\text{mid}}-B_{i-1/2}^{\text{mid}}),\\
    (r_y^+)^2-(r_y^-)^2
    &=-\Delta t\sum_{i=0}^{N_x-1}G_{y,i}^{\text{mid}}
    (B_{i+1/2}^{\text{mid}}-B_{i-1/2}^{\text{mid}}).
\end{align}
Adding the two equations leads to $\tilde{\mathcal{E}}^{n+1}_T=\tilde{\mathcal{E}}^{n}_T$ for the fourth subsystem. Therefore, we conclude the full scheme preserves the total modified energy.

The Gauss law is preserved in the system, and the proof is identical to the VA system.
\end{proof}

\begin{remark}
The extension of the SAV formulation to the 3D3V VM system is straightforward. For
$\alpha\in\{x,y,z\}$, we introduce one SAV for each component of the
electric field,
\begin{equation}
    r_\alpha^2(t)=\frac{1}{2}\int_\Omega E_\alpha^2(t,\boldsymbol{x})
    \rd{\boldsymbol{x}}+C_0,
    \qquad
    \dot r_\alpha(t)=\frac{1}{2r_\alpha(t)}\int_\Omega
    E_\alpha\big((\nabla\times\boldsymbol{B})_\alpha-J_\alpha\big)
    \rd{\boldsymbol{x}}.
\end{equation}
The corresponding modified energy is
\begin{equation}
    \widetilde{\mathcal{E}}_T
    =\frac{1}{2}\sum_{p=1}^{N_p}\omega_p|\boldsymbol{v}_p|^2
    +\sum_{\alpha}r_\alpha^2-3C_0
    +\frac{1}{2}\int_\Omega|\boldsymbol{B}|^2\rd{\boldsymbol{x}}.
\end{equation}

Scheme VM-I extends by using the three-dimensional predictor
$\boldsymbol{v}_p^*-\boldsymbol{v}_p^n=\frac{\Delta t}{2}
(\boldsymbol{E}_p^n+\boldsymbol{v}_p^*\times\boldsymbol{B}_p^n)$ and then
applying the same midpoint field, velocity, and SAV updates.
For Scheme VM-II, we set $g_\alpha=E_\alpha/r_\alpha$ and apply the same symmetric splitting.

For the current used for charge conservation, we can adopt the current deposition from Esirkepov \cite{esirkepov-2001}. On a Yee's lattice, the current
$\boldsymbol{\overline{J}}_h^{n+1/2}$ is constructed to satisfy the discrete Gauss law
\begin{equation}
    \frac{\rho_{i,j,k}^{n+1}-\rho_{i,j,k}^{n}}{\Delta t}
    +D_x\overline{J}_{x,i,j,k}^{n+1/2}
    +D_y\overline{J}_{y,i,j,k}^{n+1/2}
    +D_z\overline{J}_{z,i,j,k}^{n+1/2}=0,
\end{equation}
where $D_\alpha$ is the directional finite-difference operator. 

Esirkepov's method writes 
\begin{align}
    \rho_p^{n+1}-\rho_p^n=W_{x,p}+W_{y,p}+W_{z,p},\quad D_\alpha\overline{J}_{\alpha,p}=-W_{\alpha,p}/\Delta t,
\end{align}
where the three $W_{\alpha,p}$ contain symmetric mixed products of the one-dimensional
shape-function differences. Since the discrete
operators on Yee's lattice satisfy $\nabla_h\cdot(\nabla_h\times\boldsymbol{B}_h)=0$,
taking the discrete divergence of the Amp\`ere update propagates
the discrete Gauss law from time level $n$ to $n+1$.  

For Scheme VM-II, we introduce the scaled electric fields
$g_\alpha=E_\alpha/r_\alpha$, $\alpha\in\{x,y,z\}$. 
The splitting in Subsection 
\ref{subsec:VM-II} extends to 
the 3D-3V system as follows, 
with  $\alpha\in\{x,y,z\}$:
\begin{subequations}
\begin{align}
\varphi_{\Delta t}^{[1]}\quad &
\begin{cases}
\dot{\boldsymbol{x}}_p=\boldsymbol{v}_p,\\
r_\alpha\dot{g}_\alpha=-J_\alpha,
\end{cases}\\
\varphi_{\Delta t}^{[2]}\quad &
\begin{cases}
r_\alpha\dot{g}_\alpha
=\dfrac{1}{2}g_\alpha
\displaystyle\int_\Omega g_\alpha J_\alpha\rd{\boldsymbol{x}},
\end{cases}\\
\varphi_{\Delta t}^{[3]}\quad &
\begin{cases}
\dot{\boldsymbol{v}}_p
=\displaystyle\sum_{\alpha\in\{x,y,z\}}
r_\alpha g_{\alpha,p}\boldsymbol{e}_\alpha,\\
\dot{r}_\alpha
=-\dfrac{1}{2}\displaystyle\int_\Omega
g_\alpha J_\alpha\rd{\boldsymbol{x}},
\end{cases}\\
\varphi_{\Delta t}^{[4]}\quad &
\begin{cases}
\dot{\boldsymbol{v}}_p
=\boldsymbol{v}_p\times\boldsymbol{B}_p,\\
\dot{\boldsymbol{B}}
=-\nabla\times
\left(\displaystyle\sum_{\alpha\in\{x,y,z\}}
r_\alpha g_\alpha\boldsymbol{e}_\alpha\right),\\
\dot{r}_\alpha
=\dfrac{1}{2}\displaystyle\int_\Omega
g_\alpha(\nabla\times\boldsymbol{B})_\alpha
\rd{\boldsymbol{x}},\\
r_\alpha\dot{g}_\alpha
=(\nabla\times\boldsymbol{B})_\alpha
-\dfrac{1}{2}g_\alpha
\displaystyle\int_\Omega
g_\alpha(\nabla\times\boldsymbol{B})_\alpha
\rd{\boldsymbol{x}},
\end{cases}
\end{align}
\end{subequations}
where $\boldsymbol{e}_\alpha$ are the Cartesian unit vectors, and
$g_{\alpha,p}$ and $\boldsymbol{B}_p$ denote the fields interpolated to
particle positions. In each subsystem, the variables that are not updated
remain constant. A symmetric second-order composition is then
\begin{equation}
\varphi_{\Delta t/2}^{[1]}
\circ\varphi_{\Delta t/2}^{[2]}
\circ\varphi_{\Delta t/2}^{[3]}
\circ\varphi_{\Delta t}^{[4]}
\circ\varphi_{\Delta t/2}^{[3]}
\circ\varphi_{\Delta t/2}^{[2]}
\circ\varphi_{\Delta t/2}^{[1]}.
\end{equation}

\end{remark}

\section{Numerical results}

In this section, we present several numerical results obtained from the proposed schemes: scheme VA-I, scheme VA-II, scheme VM-I, and scheme VM-II.

\subsection{The Vlasov-Amp\`ere system}

We test the scheme VA-I and scheme VA-II on the Landau damping test cases. The initial distribution function is given by 
\begin{equation} \label{Landau damping}
    f_{LD}(x,v)=\frac{1}{\sqrt{2\pi}}e^{-\frac{v^2}{2}}(1+\alpha \cos(kx)),
\end{equation}
where $k=0.5$, $\alpha=0.1$ for weak Landau damping and $k=0.5$, $\alpha=0.5$ for strong Landau damping. We set $N_p=2\times10^5$, $N_x=100$, $L=\frac{2\pi}{k}$, $\Delta t = 0.01$ and the SAV constant $C_0=1$. The weak Landau damping is run up to $t=20$, while the strong Landau damping is run up to $t=40$. With this initial condition, the electric energy $\mathcal{E}_E(t)=\frac{1}{2}\|E(t)\|_{L^2}^2$  for the weak Landau damping has a damping rate of $\gamma \approx -2\times0.1533$; $\mathcal{E}_E(t)$ for the strong Landau damping has a damping rate of $\gamma_1 \approx -2\times0.292$ and a growth rate of $\gamma_2 \approx 2\times0.0815$. 

Figure \ref{fig:Electrostatic} shows the numerical results of the Landau damping tests, where we plot the time history of the original electric energy ${\mathcal E}_E^n=\frac{1}{2}\sum_{i=0}^{N_x-1}(E_{i+1/2}^n)^2 \Delta x$, the modified electric energy $\tilde{\mathcal E}_E^n=(r^n)^2-C_0$, the original total energy error ${\mathcal E}_E^n+\frac{1}{2}\sum_{p=1}^{N_p}\omega_p(v_p^n)^2-{\mathcal E}_E^0-\frac{1}{2}\sum_{p=1}^{N_p}\omega_p(v_p^0)^2$, the modified total energy error $\tilde{\mathcal E}_E^n+\frac{1}{2}\sum_{p=1}^{N_p}\omega_p(v_p^n)^2-\tilde{\mathcal E}_E^0-\frac{1}{2}\sum_{p=1}^{N_p}\omega_p(v_p^0)^2$, and the Gauss law error $\epsilon_G^n=\sum_{i=0}^{N_x-1}(\frac{E_{i+1/2}^n-E_{i-1/2}^n}{\Delta x} -\rho_i^n+1)^2\Delta x$ with the periodic condition $E_{-1/2}=E_{N_x-1/2}$. Both schemes agree well with the theoretical growth and damping rates, and the modified total energy and the Gauss law errors remain at machine precision. The run time for VA-II is roughly twice that of VA-I.

To demonstrate the second-order temporal convergence of VA-I and VA-II, we plot the maximum total energy error $\|\mathcal{E}_T^n-\mathcal{E}_T^0\|_{\infty}$ for time steps $\Delta t=0.01, 0.02, 0.04, 0.08, 0.16, 0.32$. Since neither scheme preserves the original total energy, we expect this error to reflect the temporal discretization error. Indeed, Figure \ref{fig:landau_time_convergence} shows the second-order convergence for both schemes. For small \(\Delta t\), however, the observed convergence rate becomes unreliable because spatial discretization or particle approximation errors may dominate.

\begin{figure}
    \centering    \includegraphics[width=1.0\textwidth]{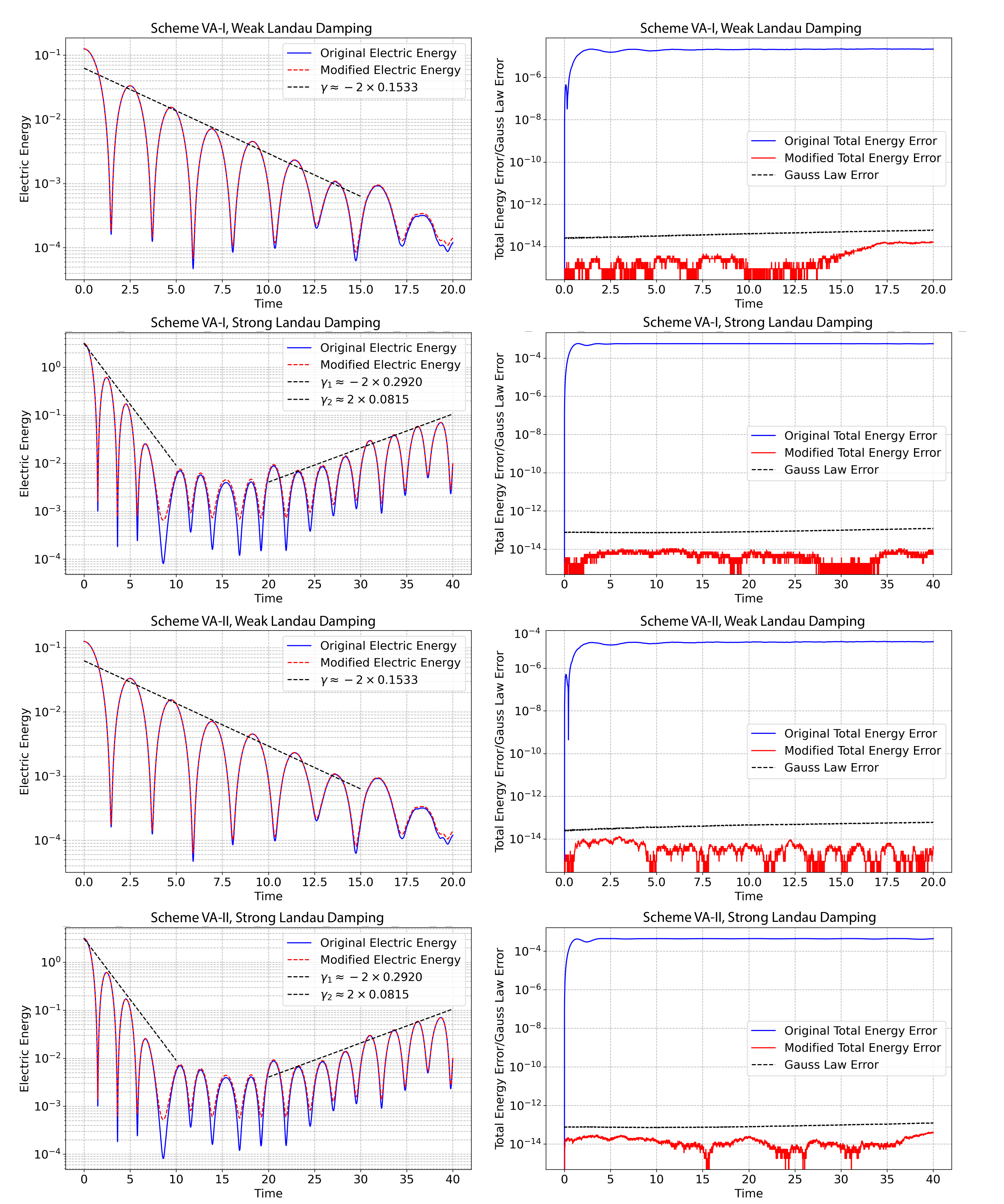}
    \caption{Left column: time evolution of $\mathcal{E}_E(t)$ and $\tilde{\mathcal{E}}_E(t)$. Upper two panels: scheme VA-I, lower two panels: scheme VA-II. 
    Right column: time evolution of $\left|\mathcal{E}_T(t) - \mathcal{E}_{T}(0)\right|$, $\left|\tilde{\mathcal{E}}_T(t) - \tilde{\mathcal{E}}_{T}(0)\right|$ and $\epsilon_G$. Upper two panels: scheme VA-I, lower two panels: scheme VA-II.}
    \label{fig:Electrostatic}
\end{figure}

\begin{figure}
    \centering     \includegraphics[width=1.0\textwidth]{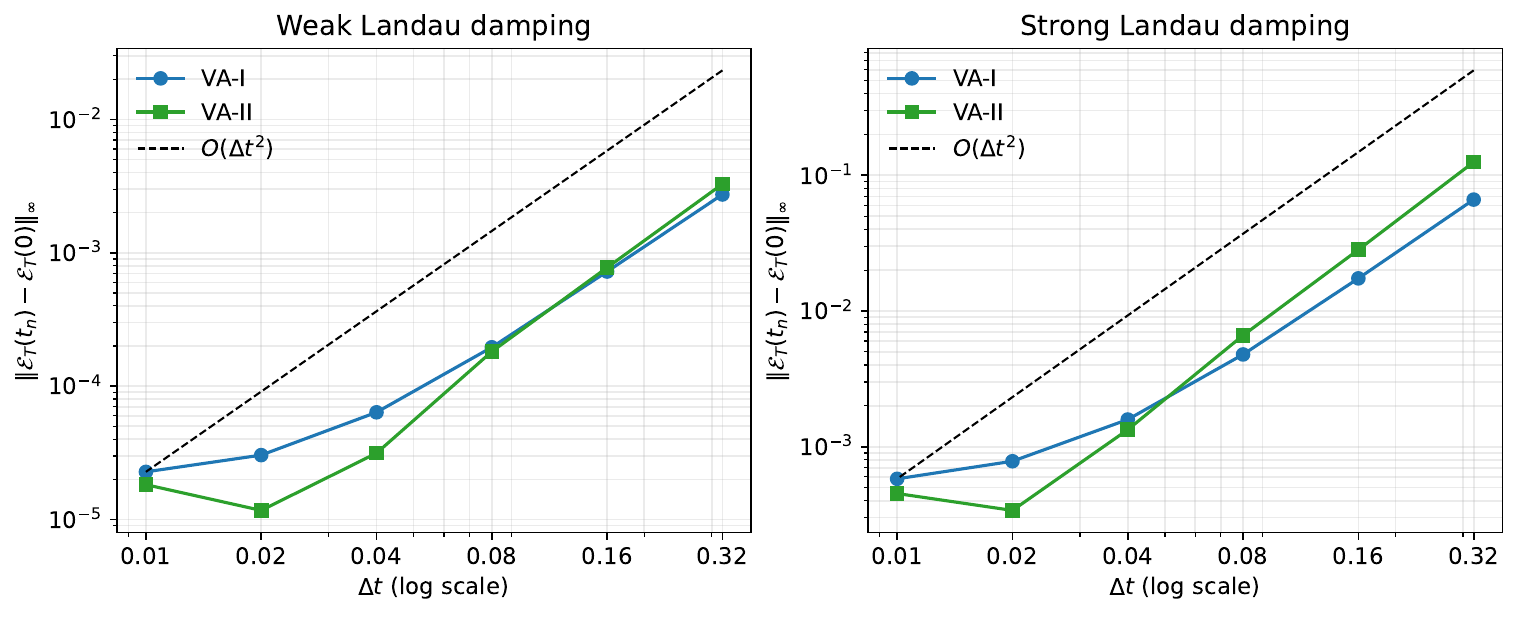}
    \caption{The maximum total energy error $\|\mathcal{E}_T^n-\mathcal{E}_T^0\|_{\infty}$ in the Landau damping tests for VA-I, VA-II. }
    \label{fig:landau_time_convergence}
\end{figure}

To study the effect of charge conservation, we adopt the cold two-stream instability test from \cite{chen-2011}. The initial distribution function is given by
\begin{equation}
    f_{TS}(x,v)=\frac{1}{2}\left[\delta(v-v_b)+\delta(v+v_b)\right](1+\alpha\cos x),
\end{equation}
where $\delta$ is the Dirac delta function. We set $L=2\pi$, $v_b=\sqrt{3}/2$, $\alpha=0.005$, $N_p=10^4$, $N_x=50$, $\Delta t=0.2$, $T=200$, and $C_0=1$. The particles are initialized with a deterministic quiet start: equally spaced quantiles of the perturbed spatial density are paired at the same position with velocities $v_b$ and $-v_b$, thereby suppressing finite-particle sampling noise at $t=0$. In addition to schemes VA-I and VA-II, we use two other schemes, denoted VA-I non-charge-conserving and VA-II non-charge-conserving, which are identical to VA-I and VA-II except that $\bar{J}_{i+1/2}$  used in \eqref{VA1_E} and \eqref{VA-II_current} is replaced by the midpoint current
\begin{equation}
    J_{i+1/2}^{*,n+1/2}=\frac{1}{\Delta x}\sum_{p=1}^{N_p}\omega_p v_p^* S^m(x_{i+1/2}-x_p^{n+1/2}),
\end{equation}
which maintains second-order temporal accuracy but violates the discrete Gauss's law.

Figure \ref{fig:charge-conservation} compares the four schemes. The electric energy in all schemes aligns with the theoretical growth rate $\gamma=0.5$ initially, but VA-I non-charge-conserving develops an instability in the nonlinear phase. At later times, it has a much larger Gauss law error and a larger maximum momentum error. VA-II non-charge-conserving still shows a reasonable electric energy error, suggesting that violating Gauss law may not be so vital for the VA-II scheme.

\begin{figure}[!htbp]
    \centering
    \includegraphics[width=0.82\textwidth]{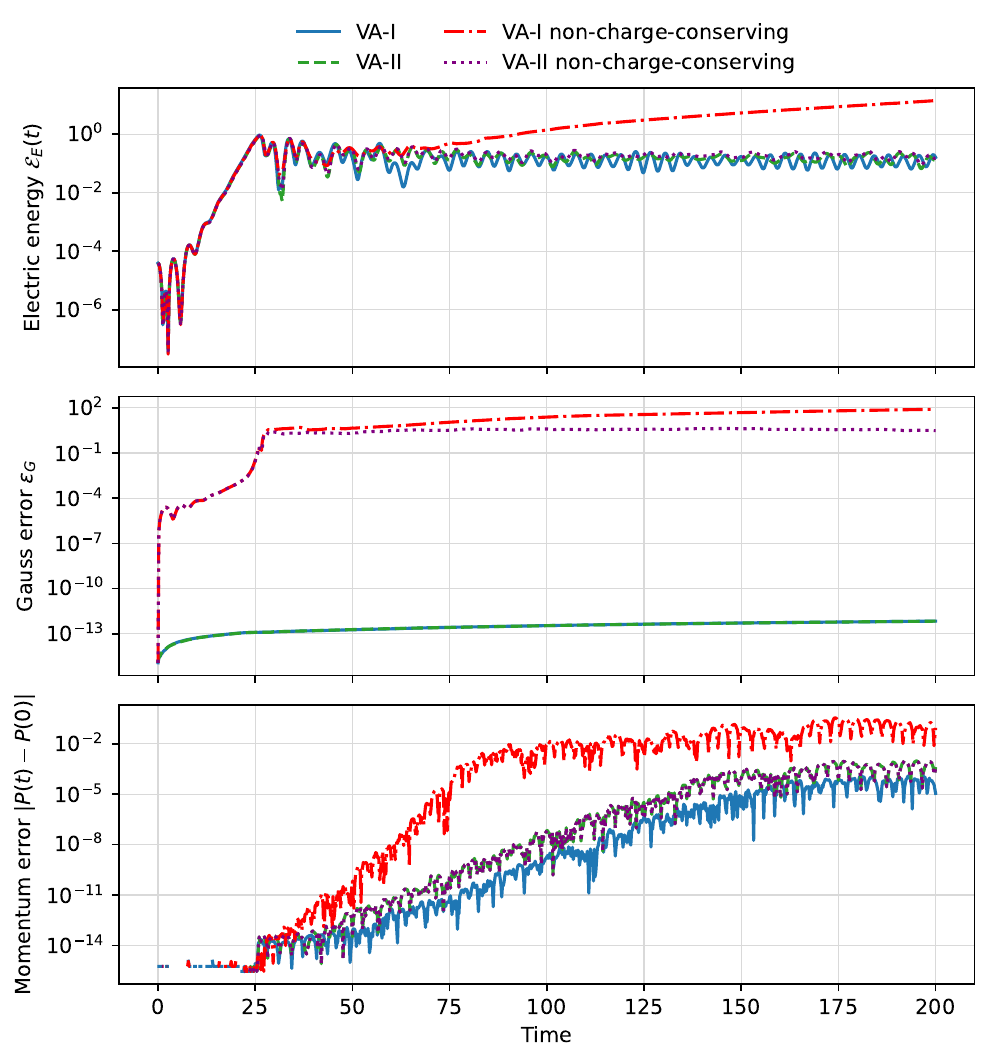}
    \caption{Two-stream instability test. Comparison of VA-I, VA-II, VA-I non-charge-conserving, and VA-II non-charge-conserving. Time evolution of $\mathcal{E}_E$, $\epsilon_G$ and momentum error $|P(t)-P(0)|$.}
    \label{fig:charge-conservation}
\end{figure}

\subsection{The Vlasov-Maxwell system}

We test the scheme VM-I and scheme VM-II on the Weibel instability test cases, with the initial distribution and fields given by
\begin{align} \label{Weibel instability}
    f_{WI}(x,v_x,v_y) & =\frac{1}{2\pi \sigma_1 \sigma_2}\exp\Big(-\frac{v_x^2}{2\sigma_1^2}-\frac{v_y^2}{2\sigma_2^2}\Big), \\
    B(x)         & =b_0 \sin(kx),
\end{align}
where $\sigma_1=0.02/\sqrt{2}$, $\sigma_2=\sqrt{12} \sigma_1$, $b_0=10^{-4}$ and $k=1.25$. We set $N_p=2\times10^5$, $N_x=100$, $L=\frac{2\pi}{k}$, $\Delta t = 0.05$ and the SAV constant $C_0=1$. The simulation is run up to $t=300$. With this initial condition, the magnetic field energy $\mathcal{E}_B(t)=\frac{1}{2}\|B(t)\|_{L^2}^2$ has a theoretical growth rate of $\gamma_0\approx2\times0.02784$ \cite{weibel-1959}. 

Figure \ref{fig:Electromagnetic} shows the numerical results of the tests, where we plot the time history of the electric energy in $x$ direction ${\mathcal E}_{E_x}^n=\frac{1}{2}\sum_{i=0}^{N_x-1}(E_{x,i+1/2}^n)^2 \Delta x$, the electric energy in $y$ direction ${\mathcal E}_{E_y}^n=\frac{1}{2}\sum_{i=0}^{N_x-1}(E_{y,i}^n)^2 \Delta x$, the magnetic energy $\mathcal{E}_B^n=\frac{1}{2}\sum_{i=0}^{N_x-1} (B^n_{i+1/2})^2 \Delta x$, the original total energy error ${\mathcal E}_{E_x}^n+{\mathcal E}_{E_y}^n+\mathcal{E}_B^n- {\mathcal E}_{E_x}^0 - {\mathcal E}_{E_y}^0-\mathcal{E}_B^0$, the modified total energy error $(r_x^n)^2+(r_y^n)^2-2C_0+{\mathcal E}_{E_y}^n+\mathcal{E}_B^n- {\mathcal E}_{E_x}^0 - {\mathcal E}_{E_y}^0-\mathcal{E}_B^0$, and the Gauss law error $\epsilon_G^n=\sum_{i=0}^{N_x-1}(\frac{E_{x,i+1/2}^n-E_{x,i-1/2}^n}{\Delta x} -\rho_i^n+1)^2\Delta x$ with the periodic condition $E_{x,-1/2}=E_{x,N_x-1/2}$. Both schemes agree well with the theoretical growth rate, and the modified total energy and the Gauss law error are preserved up to machine precision. The run time for VM-II is roughly twice that of VM-I.

\begin{figure}
    \centering
    \includegraphics[width=1.0\textwidth]{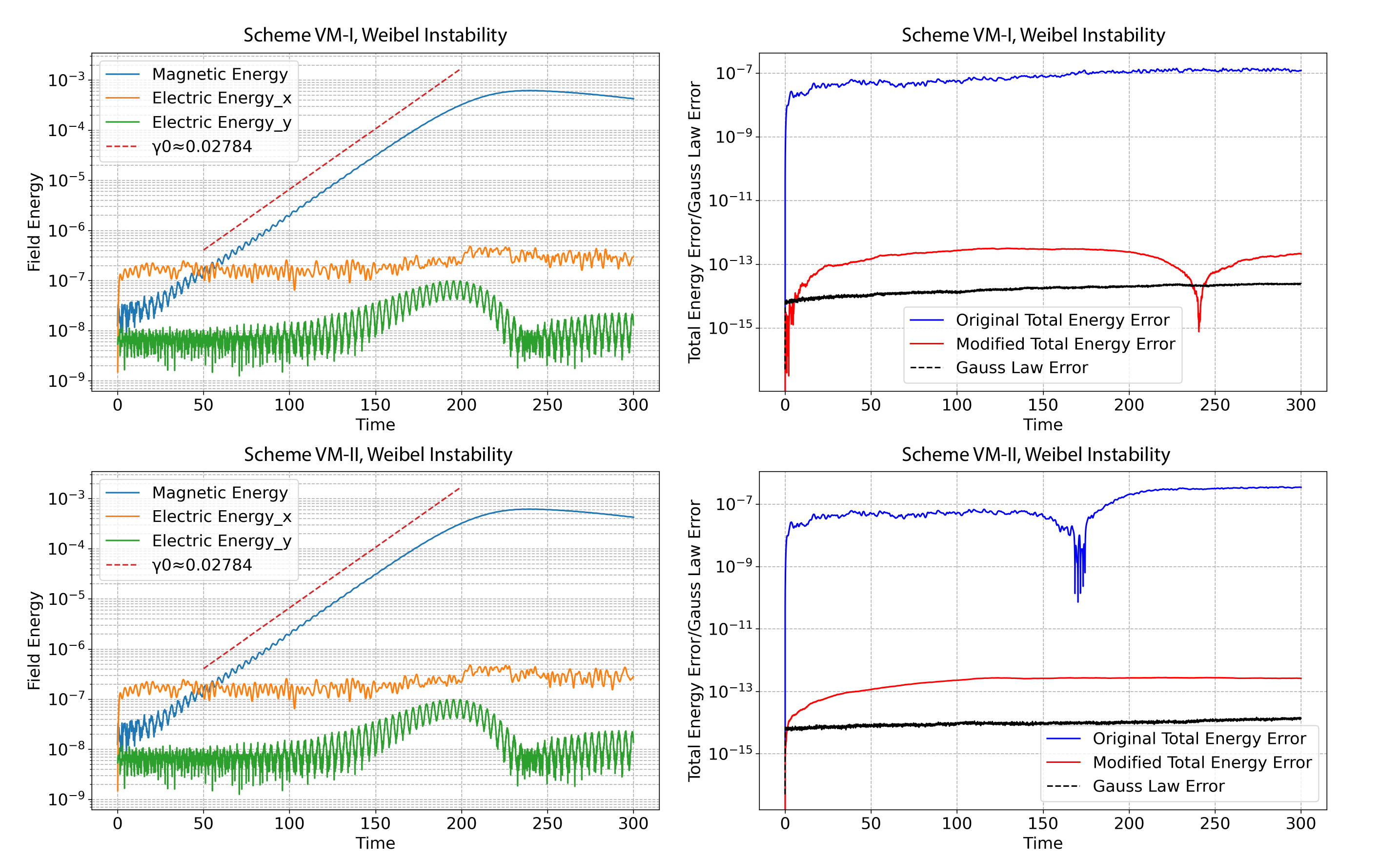}
    \caption{Left column: time evolution of $\mathcal{E}_{E_x}(t)$, $\mathcal{E}_{E_y}(t)$ and $\mathcal{E}_{B}(t)$. Upper panel: scheme VM-I, lower panel: scheme VM-II. Right column: time evolution of $\left|\mathcal{E}_T(t) - \mathcal{E}_{T}(0)\right|$, $\left|\tilde{\mathcal{E}}_T(t) - \tilde{\mathcal{E}}_{T}(0)\right|$ and $\epsilon_G$. Upper panel: scheme VM-I, lower panel: scheme VM-II.}
    \label{fig:Electromagnetic}
\end{figure}

\section{Conclusions and future work}

In this article, we proposed SAV-based reformulated Vlasov--Amp\`ere system and Vlasov-Maxwell system and two different numerical schemes to approximate each system. We used the scalar auxiliary variable as a degree of freedom to preserve a modified energy, with the local charge being preserved thanks to a suitable current deposition technique. 
In the first  scheme (VA-I and VM-I), inspired by \cite{ricketson2025}, an additional update of the scalar variable enables the preservation of the modified energy. In the second scheme (VA-I and VM-II), 
the scalar variable is included in the pusher but the electric field is solved a posteriori. The resulting schemes are second order accurate in time, 
linearly implicit in the fields and explicit in the particles. They are free from the Maxwell's CFL condition and both preserve
exactly the discrete charge and the modified total energy. The numerical experiments illustrate the capability of the schemes to reproduce standard tests from the plasma physics literature. Future work includes further investigation of the differences between these SAV schemes and comparisons with other SAV schemes.

\section*{Acknowledgement}
The work of CF and JH was partially supported by NSF grants DMS-2409858 and IIS-2433957, DOE DE-SC0023164, and DoD MURI FA9550-24-1-0254. 
The work of NC has been carried out within the framework of the EUROfusion Consortium, funded by the European Union via the Euratom Research and Training
Programme (grant agreement No 101052200 EUROfusion). Views and opinions expressed are however those of the authors only and do not necessarily reflect those of the European Union or the European Commission.

\bibliographystyle{unsrturl}
\bibliography{main}

\end{document}